\documentclass[12pt,a4paper]{amsart}
\usepackage[centering, margin=1in]{geometry}
\usepackage[latin1]{inputenc}
\usepackage[english]{babel}
\usepackage[T1]{fontenc}
\usepackage{amsmath}
\usepackage{amssymb}
\usepackage{amsthm}
\usepackage{hyperref}
\usepackage{array,booktabs}
\usepackage{graphicx}
\usepackage{newtxtext}
\usepackage{quoting}
\usepackage{enumerate}
\usepackage{mathtools}
\usepackage{csquotes}
\usepackage{cite}
\usepackage{tikz}
\usepackage{tikz-cd}
\usetikzlibrary{patterns}
\usetikzlibrary{decorations.markings}
\usetikzlibrary{math}
\usetikzlibrary{matrix,arrows,decorations.pathmorphing}
\usepackage{wrapfig}
\renewcommand{\Re}{\operatorname{Re}}
\renewcommand{\Im}{\operatorname{Im}}
\def\R{\ensuremath\mathbb{R}}
\def\C{\ensuremath\mathbb{C}}
\def\Z{\ensuremath\mathbb{Z}}
\def\Q{\ensuremath\mathbb{Q}}

\def\Hb{\ensuremath\mathbb{H}}

\usepackage{hyperref}
\usepackage[final]{pdfpages}
\usepackage{verbatim}
\newtheorem{thm}{Theorem}[section]
\newtheorem{defi}[thm]{Definition}
\newtheorem{cor}[thm]{Corollary}
\newtheorem{lemma}[thm]{Lemma}
\newtheorem{prop}[thm]{Proposition}

\theoremstyle{remark}
\newtheorem{remark}[thm]{Remark}
\newtheorem{example}[thm]{Example}

\def\eps{\ensuremath\varepsilon}
\def\0{\emptyset}

\def\mod {\text{ {\rm mod} }}

\def\PSL{\mathrm{PSL}}

\def\CC{\mathcal{C}}
\def\FF{\mathcal{F}}
\def\Cc{\mathbf{C}}
\def\Cl{\mathrm{Cl}}
\def\NN{\mathcal{N}}

\def\supp{\mathrm{supp}}
\def\BB{\mathcal{B}}

\def\LL{\mathcal{L}}
\def\HH{\mathcal{H}}
\def\PP{\mathcal{P}}
\def\AA{\mathcal{A}}
\def\cc{\mathbf{c}}
\def\Hom{\mathrm{Hom}}
\def\modulo{\text{ \rm mod }}
\def\vol{\text{\rm vol}}

\numberwithin{equation}{section}
\begin{document}
\author{Asbj\o{}rn Christian Nordentoft}

\address{Department of Mathematical Sciences, University of Copenhagen, Universitetsparken 5, 2100 Copenhagen \O, Denmark}
\email{\href{mailto:nordentoft@math.ku.dk}{nordentoft@math.ku.dk}}

\urladdr{\href{https://sites.google.com/view/asbjornnordentoft/}{https://sites.google.com/view/asbjornnordentoft/}} 

\author{Ser Peow Tan}

\address{Department of Mathematics, National University of Singapore, 10 Lower Kent Ridge Road, 119076, Singapore}

\email{\href{mailto:mattansp@nus.edu.sg}{mattansp@nus.edu.sg}}

\urladdr{\href{https://sites.google.com/view/serpeowtan/home}{https://sites.google.com/view/serpeowtan/home}}

\title{Equidistribution of partial coverings from closed geodesics}
\begin{abstract}
Given a finite volume hyperbolic surface, a fundamental polygon and an oriented closed geodesic, we associate a partial covering of the surface. 
We prove that given a sequence of collections of oriented closed geodesics equidistributing in the unit tangent bundle then the associated partial coverings also equidistribute. In the case of the modular group, this yields an alternative proof of an equidistribution result due to Duke, Imamo\={g}lu, and T\'{o}th.
\end{abstract}
\maketitle
\section{Introduction}
The distribution of closed geodesics on hyperbolic surfaces is a classical and rich topic with connections to ergodic theory \cite{Bowen72}, \cite{Li68}, \cite{Skubenko62}, the spectral theory of automorphic forms \cite{Zelditch92}, number theory \cite{Du88}, \cite{HumphriesRadziwill22}, \cite{NordentoftConcentration} and hyperbolic geometry \cite{ErlandssonSouto23}, \cite{ErlandssonSouto22}. In the case of arithmetic quotients the closed geodesics are naturally parametrized by integral indefinite binary quadratic forms and thus in turn by class groups of orders of real quadratic fields \cite{Sarnak82}.  In \cite{DuImTo} Duke--Imamo\={g}lu--T\'{o}th introduced a new geometric invariant, a hyperbolic orbifold $\Gamma_A\backslash \mathcal{N}_A$ with boundary, associated to (narrow) ideal classes $A\in \Cl_K^+$ of a real quadratic field $K$. The boundary of this orbifold is  given by the corresponding closed geodesic $\partial (\Gamma_A\backslash \mathcal{N}_A)=\mathcal{C}_A$ on the modular surface $\PSL_2(\Z)\backslash \Hb$. They furthermore proved that the hyperbolic orbifolds associated to genera of the class groups (i.e.\ cosets of the square subgroups $(\Cl_K^+)^2$) equidistribute when projected to the modular surface as the discriminant of $K$ tends to infinity. Their proof is spectral in nature and relies on relating the Weyl sums to Fourier coefficients of half-integral weight automorphic forms which have been bounded by Duke \cite{Du88}. The construction was generalized to Hecke congruence groups of prime level by Peter Humphries and the first named author \cite{HumphriesNordentoft22} using geometric coding of closed geodesics and certain refined sparse equidistribution questions were investigated. The proofs in \emph{loc.\ cit.}\ rely on an adelic interpretation of the Weyl sums which are then related to central values via Waldspurger's formula \cite{Waldspurger85}. 

In this paper we introduce a new purely geometric construction of the partial coverings associated to closed geodesics alluded to above. The construction generalizes to an arbitrary finite volume hyperbolic surface and our main result (Theorem \ref{thm:mainintro}) is that the partial coverings \emph{cover the hyperbolic surface evenly}, in the sense of eq. (\ref{eq:coverevenly}), if the corresponding oriented closed geodesics \emph{equidistribute in the unit tangent bundle}, in the sense of Definition \ref{def:equid}. In particular, it follows that the equidistribution result of Duke--Imamo\={g}lu--T\'{o}th \cite[Thm. 2]{DuImTo} is a consequence of the equidistribution result of Duke\footnote{More precisely, it follows from the extension to genera of the class group and to the unit tangent bundle as in e.g. \cite{EinLindMichVenk12}.} \cite{Du88} as explained in Remark \ref{rem:DITHumNor} and Section \ref{sec:comparing} below. For an overview of the proof of Theorem \ref{thm:mainintro} we refer to Section \ref{sec:ideaofproof} below.

\subsection{Statement of results}
In order to state our results, we will now give an informal description of the construction of \emph{partial coverings} obtained by ``painting to the left'' of oriented closed geodesics relative to a fixed fundamental polygon. We refer to Section \ref{sec:PC} below for details and to Example \ref{ex:triangle} for illustrations of the construction in the case of the $(2,4,6)$ triangle group.           

Let $\Gamma\leq \PSL_2(\R)$ be a discrete and cofinite subgroup. Consider the associated hyperbolic orbifold $Y_\Gamma:=\Gamma\backslash \Hb$ and its unit tangent bundle $\Gamma \backslash \PSL_2(\R)$. Here $\Hb=\{z\in \C: \Im z>0\}$ denotes the hyperbolic plane. Let $\FF\subset \Hb$ be a fundamental polygon for $\Gamma$. Let $\CC$ be an oriented closed geodesic on $Y_\Gamma$ and let $\tilde{\CC}\subset \Hb$ be a lift to the universal cover. Then for each $\Gamma$-translate $\gamma_i \FF$ of $\FF$ which the lift $\tilde{\CC}$ intersects we obtain a hyperbolic polygon $\PP_i$ as the intersection of $\gamma_i \FF$ with the interior (relative to the orientation of $\tilde{\CC}$ coming from $\CC$) of the two-sided infinite geodesic containing $\tilde{\CC}$, see eq. (\ref{eq:PPi}) for a precise definition. This is what we informally think of as ``painting to the left of $\CC$ relative to $\FF$\,'' and is depicted in Figures \ref{fig:1} and \ref{fig:2} for the $(2,4,6)$ triangle group. 

The collection of polygons $\cup_i \PP_i$ defines a ``partial covering'' of $Y_\Gamma$ by considering the image of $\cup_i \PP_i$ under the natural projection  $\pi_\Gamma:\Hb\rightarrow Y_\Gamma$ counted with multiplicities. Formally, we will encode this in the function:
\begin{equation}\PP(\CC,\FF): Y_\Gamma\rightarrow \Z_{\geq 0},\quad x\mapsto \#(\pi_\Gamma^{-1}(\{x\})\cap (\cup_i \PP_i)) ,\end{equation}
which we refer to as the \emph{partial covering of $Y_\Gamma$ associated to $(\CC,\FF)$}. This is depicted in Figure \ref{fig:3}. 
For a collection of oriented closed geodesics $\Cc=\{\CC_1,\ldots, \CC_k\}$ we define the associated partial covering as   
\begin{equation}\PP(\Cc,\FF):=\sum_{i=1}^k \PP(\CC_i,\FF),\end{equation}
with addition defined on the target. Given a sequence of collections of geodesics $(\Cc_n)_{n\geq 1}$ the basic question is if the associated partial coverings ``cover $Y_\Gamma$ evenly'' as $n\rightarrow \infty$, corresponding to the partial coverings $\PP(\Cc_n,\FF)$ thought of as a functions being ``asymptotically constant''. To make this precise, for an integrable function $\phi:Y_\Gamma\rightarrow \C$ we define:
$$\int_{\PP(\Cc_n,\FF)}\phi:= \int_{Y_\Gamma} \phi(z)\PP(\Cc_n,\FF)(z) d\mu(z),  $$
where $d\mu(z)=\tfrac{dxdy}{y^2}$ denotes the hyperbolic measure. We say that 
$$\PP(\Cc_n,\FF) \text{\emph{ covers $Y_\Gamma$ evenly as $n\rightarrow\infty$}},$$ if for any continuous and bounded function $\phi:Y_\Gamma\rightarrow \C$ it holds that 
\begin{equation}\label{eq:coverevenly}\frac{\int_{\PP(\Cc_n,\FF)}\phi}{\vol(\PP(\Cc_n,\FF))}\rightarrow  \frac{\int_{Y_\Gamma} \phi}{\vol(\Gamma)},\quad \text{as }n\rightarrow \infty. \end{equation}
Here $\vol(\PP(\Cc_n,\FF)):=\int_{\PP(\Cc_n,\FF)}1$ and $\vol(\Gamma):=\int_{Y_\Gamma} 1$.
Our main result is the following:

\begin{thm}\label{thm:mainintro}
Let $\FF$ be a fundamental polygon for a discrete and cofinite subgroup $\Gamma\leq \PSL_2(\R)$. Let $\Cc_1,\Cc_2,\ldots$ be a sequence of collections of oriented closed geodesics on $Y_\Gamma$ equidistributing in the unit tangent bundle $\Gamma\backslash \PSL_2(\R)$, in the sense of Definition \ref{def:equid}\footnote{Meaning, converging weakly on the space of continuous and bounded functions to the probability Liouville (or Haar) measure on $\Gamma\backslash \PSL_2(\R)$.}. Then $\PP(\Cc_n,\FF)$ covers $Y_\Gamma$ evenly as $n\rightarrow \infty$.
\end{thm} 
\begin{remark}
In the co-compact case we also obtain an effective version of the theorem, see Theorem \ref{thm:main}. In the non-compact case there are subtleties to be resolved at the cusps, see Remark \ref{remark:effective} for details. 
\end{remark}
\begin{remark} The notion of a sequence of partial coverings $(\PP_n)_{n\geq 1}$ ``to cover $Y_\Gamma$ evenly'' in the case where $\PP_n$ are coverings obtained as the projection of some (disjoint) polygons $\subset \Hb$ (the union of which we will denote by the same symbol $\PP_n$) amounts to showing for any continuity set $\mathcal{B}\subset \Gamma\backslash\Hb$ that  
\begin{equation}
\label{eqn:equidistribution}
\frac{\vol(\PP_n \cap \Gamma \mathcal{B})}{\vol(\PP_n)} \rightarrow  \frac{\vol(\mathcal{B})}{\vol(\Gamma)},\quad \text{as }n\rightarrow \infty.
\end{equation}
This follows by unraveling the definitions and approximating smooth and compactly supported functions by indicator functions (and visa versa), see Lemma \ref{lem:B}. The statement (\ref{eqn:equidistribution}) is exactly the notion of equidistribution used in both \cite{DuImTo} and \cite{HumphriesNordentoft22}.
\end{remark}
\begin{remark}
In view of  Lemma \ref{lem:DITequiv} our construction of partial coverings specializes to that of Duke--Imamo\={g}lu--T\'{o}th. Thus by combining Theorem \ref{thm:mainintro} with the ergodic proof of equidistribution of oriented closed geodesics due to Einsiedler--Lindenstrauss--Michel--Venkatesh \cite{EinLindMichVenk12} we obtain a ``theta correspondence-free'' proof of the equidistribution results \cite[Thm. 2]{DuImTo}, by which we mean a proof that avoids any use of bounds for $L$-functions and/or Fourier coefficients of half-integral weight automorphic forms. This yields a possible answer to a question \cite[Ques. 10.1]{HumphriesNordentoft22} raised by Peter Humphries and the first named author. See Sections \ref{sec:comparing} and \ref{sec:humnor} as well as  Corollary \ref{cor:itfollows} for more details.
\end{remark}
\subsubsection{Applications} 
By Theorem \ref{thm:mainintro} we have reduced the equidistribution of partial coverings to the equidistribution of oriented closed geodesics in the unit tangent bundle $\Gamma\backslash \PSL_2(\R)$. Using existing results in the literature, this allows us to obtain a variety of new equidistribution results for partial coverings. We refer to Section \ref{sec:appl} for details. We note that, as observed in \cite{DuImTo}, in order to get a non-trivial statement we will need equidistributing collections of oriented closed geodesics which are \emph{not} closed under inverting the orientation of the closed geodesics, see Remark \ref{rmk:trivial} and Lemma \ref{lem:trivial} below for details. We present here four different applications. For a collection $\Cc$ of closed geodesic on $Y_\Gamma$ we denote the total length by:
$$|\Cc|:=\sum_{\CC\in \Cc}|\CC|,$$
where $|\CC|$ denotes the \emph{geodesic length} of $\CC$ as in (\ref{eq:geodesiclength}). First of all, we obtain a result for any discrete and cofinite Fuchsian group using sparse equidistribution via ergodicity of the geodesic flow, see \cite[Thm.\ 4.5]{ConstantinescuNordentoft24}.
\begin{cor}\label{cor:intro3}
Fix $\eps>0$. Let $\FF$ be a fundamental polygon for a discrete and cofinite subgroup $\Gamma\leq \PSL_2(\R)$.  Let $L_1<L_2<\ldots$  be a sequence of real numbers such that  $L_n\rightarrow \infty$ as $n\rightarrow \infty$. For each $n\geq 1$, let
$$\Cc_n\subset\{\CC\subset Y_\Gamma: |\CC|\leq L_n\},$$   
be a ``not too thin subset'' of the packet of primitive oriented closed geodesics of length $\leq L_n$, in the sense that 
$$ |\Cc_n|\geq \eps \cdot |\{\CC\subset Y_\Gamma: |\CC|\leq L_n\}|. $$ 
Then $\PP(\Cc_n,\FF)$ covers $Y_\Gamma$ evenly as $n\rightarrow \infty$.
\end{cor}
For the modular group we can obtain results for much sparser families of closed geodesics using a result of Einsiedler--Lindenstrauss--Michel--Venkatesh \cite{EinLindMichVenk12}. 
\begin{cor}\label{cor:intro1}
Let $\FF$ be a fundamental polygon for the modular surface $Y_0(1)=\PSL_2(\Z)\backslash \Hb$. Let $\ell_1<\ell_2<\ldots $ be the primitive length spectrum of $Y_0(1)$.  For each $n\geq 1$ let
$$\Cc_n\subset \{\CC\subset Y_0(1): |\CC|=\ell_n\},
$$ 
be a ``not too thin subset'' of the length packet, in the sense that 
$$ \frac{\log |\Cc_n|}{\log (\sum_{|\CC|=\ell_n }|\CC|))}\rightarrow 1,\quad n\rightarrow \infty. $$ 
Then $\PP(\Cc_n,\FF)$ covers $Y_0(1)$ evenly as $n\rightarrow \infty$.
\end{cor}
Given a square-free integer $M\geq 1$ and a positive fundamental discriminant $D>0$ such that all prime divisors of $M$ split in $\Q(\sqrt{D})$ one can associate to each (narrow) ideal class $A\in \Cl^+_D$ an oriented closed geodesic $\CC_A(M)$ on the modular surface $Y_0(M):=\Gamma_0(M)\backslash \Hb$ of level $M$ where $\Gamma_0(M)$ denotes Hecke congruence group of level $M$. In this setting we obtain results for very sparse collections of closed geodesics defined via the algebraic structure of the class groups by applying the works of Waldspurger \cite{Waldspurger85}, Popa \cite{Popa06}, and Harcos--Michel \cite{HarcosMichel06}.
\begin{cor}\label{cor:intro2}
Fix $\delta\in [0,\frac{1}{2826})$. Let $\FF$ be a fundamental polygon for $Y_0(M)=\Gamma_0(M)\backslash \Hb$ with $M\geq 1$ square-free. Let $(D_n)_{n\geq 1}$ be a sequence of distinct positive fundamental discriminants such that all prime divisors of $M$ split in $\Q(\sqrt{D_n})$. For each $n\geq 1$ let
$$\Cc_n=\{\CC_{A}(M)\subset Y_0(M): A\in BH_n\},
$$ 
be the collection of primitive oriented closed geodesics on $Y_0(M)$ associated to the coset $BH_n$ of a subgroup $H_n\leq  \Cl_{D_n}^+$ of the narrow class group of $\Q(\sqrt{D_n})$ of index $\leq D_n^\delta$. Then $\PP(\Cc_n,\FF)$ covers $Y_0(M)$ evenly as $n\rightarrow \infty$.
\end{cor}
Finally, following the work of the first named author \cite{Nordentoft23} we obtain results for \emph{$q$-orbits of closed geodesics} on the modular surface, i.e.\ closed geodesics associated to elements of $(\Z/q)^\times$ in analogy with $q$-orbits on low lying horocycles as investigated in \cite[Sec. 2.3]{FouvryKowMich15} (stated here only for level one for simplicity).
\begin{cor}\label{cor:intro4}
Fix $\delta\in [0,\frac{1}{8})$. Let $\FF$ be a fundamental polygon for the modular surface $Y_0(1)=\PSL_2(\Z)\backslash \Hb$. For $q\geq 1$ and $a\in (\Z/q)^\times$ let $\CC_a$ be the oriented closed geodesic associated to the $\PSL_2(\Z)$-conjugacy class of the hyperbolic matrix
$$ \begin{psmallmatrix} a& \frac{ad-1}{q}\\ q& d \end{psmallmatrix}  \in \PSL_2(\Z),$$ 
where $0<a<q$ and $1<d<q+1$ satisfies $ad\equiv 1\modulo q$.
 For each $q\geq 1$ let
$$\Cc_q=\{\CC_a\subset Y_0(1): a\in bH_q\},
$$ 
be the collection of primitive oriented closed geodesics on $Y_0(1)$ associated to a coset $bH_q$ of a subgroup $H_q\leq  (\Z/q)^\times$ of index $\leq q^{\delta}$. Then $\PP(\Cc_q,\FF)$ covers $Y_0(1)$ evenly as $q\rightarrow \infty$.
\end{cor}
\begin{remark}\label{rmk:trivial}
If the collections of oriented closed geodesics on $Y_\Gamma$ are closed under inversion in the sense that $\CC\in \Cc_n\Rightarrow \overline{\CC}\in \Cc_n$ (where $\overline{\CC}$ denotes the oriented closed geodesic with the same image but opposite orientation as $\CC$), then in fact $\PP(\Cc_n,\FF)$ is a \emph{perfect} covering of $Y_\Gamma$, see Lemma \ref{lem:trivial}. Thus in this case the conclusion of Theorem \ref{thm:mainintro} is trivial. This is why we in the application above restrict to various subcollections of the different length packets. This phenomena was also present in \cite{DuImTo} where they studied equidistribution for genera instead of the entire class group. 
\end{remark}
\begin{remark}\label{rem:DITHumNor} Applying Corollary \ref{cor:intro2} to respectively, $\PSL_2(\Z)$ with one of the classical fundamental polygons (\ref{eq:Fstd}) and to $\Gamma_0(q)$ for $q$ prime with a \emph{special fundamental polygon} as defined by Kulkarni \cite{Kulkarni91}, one obtains  the equidistribution results by Duke--Imamo\={g}lu--T\'{o}th \cite[Thm. 2]{DuImTo} and by Peter Humphries and the first named author \cite[Thm. 1.3]{HumphriesNordentoft22} respectively, see  Lemmas \ref{lem:DITequiv} and \ref{lem:HuNoequiv} and Corollary \ref{cor:itfollows} below. Note however that the results obtained in the present paper do not say anything about the case where the group $\Gamma$ is varying as in  \cite[Thm. 1.2]{HumphriesNordentoft22}.\end{remark}
\begin{remark}
Using existing techniques one should be able to extend Corollary \ref{cor:intro2} to general arithmetic subgroups of $\PSL_2(\R)$ associated to Eichler orders in indefinite quaternion algebras over $\Q$. This does however not seem to be available in the literature.  
\end{remark}   
\subsection{Idea of proof}\label{sec:ideaofproof} By a standard approximation argument (see Lemma \ref{lem:Zelditch}) we are reduced to proving the convergence (\ref{eq:coverevenly}) for each smooth and compactly supported function $\phi: Y_\Gamma \rightarrow \C$. This requires on the one hand, to control the volume of the partial covering and on the other hand, to handle the integral over $\PP(\Cc_n,\FF)$. A new insight in the present work is that for purely geometric reasons the equidistribution of closed geodesics in the unit tangent bundle implies a lower bound on the volume: 
$$\vol(\PP(\Cc_n,\FF))\gg |\Cc_n|,$$
see Proposition \ref{prop:vol}. The idea being that equidistribution ensures that the geodesics do not spend too much time near the boundary of $\FF$. In \cite[eq. (6.4)]{DuImTo} and \cite[Prop. 3.14]{HumphriesNordentoft22} such a lower bound was achieved for partial coverings of individual geodesics using specific properties of the fundamental polygons considered. The starting point in dealing with the integral in (\ref{eq:coverevenly}), going back to \cite[Lem. 2]{DuImTo}, is that by an application of Stokes' theorem one obtains an integral over the boundary $\partial \PP(\Cc_n,\FF)$. More precisely, if $L_2=1+2iy\partial_{\overline{z}}$ denotes weight $2$ lowering operator (\ref{eq:levelraising}) and we can write $\phi= L_2 \Phi$ for $\Phi$ an automorphic form for $\Gamma$ of weight $2$ then we have
$$ \int_{\PP(\Cc_n,\FF)} \phi=\int_{\partial\PP(\Cc_n,\FF)} \Phi, $$
see Lemma \ref{lem:Stokes}. Using the spectral theory of the Laplacian and Poincar\'{e} series in the compact and non-compact case, respectively, we show that $\phi= L_2 \Phi$ can be solved with good control on the regularity exactly if $\langle \phi,1\rangle=0$ (see Lemmas \ref{lem:antiderivcpt} and \ref{lem:ODE}). Thus in this case we are left with an integral over the boundary which has two contributions: the closed geodesics $\Cc_n$ themselves and the contribution $\LL$ from the boundary $\partial\FF$ of the fundamental polygon $\FF$ which we refer to as the \emph{topological terms}. Now the integral over $\Cc_n$ is $o(|\Cc_n|)$ by the equidistribution assumption. The second key new insight is that the topological terms can be controlled for \emph{any} $\FF$ in terms of the homology of the compactification of  $Y_\Gamma$ (see Section \ref{sec:topology}). Furthermore, $\LL$ is homologous to $-\Cc_n$ and so by the de Rham theorem the topological terms can be related to integrals of weight $2$ holomorphic cusp forms over the closed geodesics $\Cc_n$ which similarly are $o(|\Cc_n|)$ by the equidistribution assumption (as it corresponds to testing with  weight $2$ automorphic forms, see Lemma \ref{lem:weight2}). The above argument yields the convergence (\ref{eq:coverevenly}) when $\langle \phi,1\rangle=0$. When $\Gamma$ is co-compact the general case follows immediately by subtracting the projection to the constant. In the non-compact case we need to furthermore ensure non-escape of mass (since the constant is not of compact support). This is achieved by a geometric argument comparing hyperbolic lengths and volumes, again relying on the assumption of equidistribution (see Proposition \ref{prop:nonescape}).  
\section*{Acknowledgements}
The authors would like to thank Peter Humphries, Sanghyun Kim, Nhat Minh Doan, Tien-Cuong Dinh and Johannes Bartenschlager for useful discussions related to the topic of this paper. We would also like to give a special thanks to the referee for his/her comments which raised the level of the paper. The second named author was partially supported by the National University of Singapore academic research grant A-8000989-00-00. 


\section{Background on automorphic forms}
We refer to \cite[Sec. 4]{DuFrIw02} and \cite{Iw} for a comprehensive introduction to the following material.
\subsection{Cofinite hyperbolic orbifolds}
The group $G:=\PSL_2(\R)$ acts on hyperbolic space $\Hb:=\{z\in \C: \Im z>0\}$ by hyperbolic isometries via linear fractional transformations 
$$\gamma z:= \frac{az+b}{cz+d},\quad z\in \Hb, \gamma=\begin{pmatrix} a&b\\c&d\end{pmatrix}\in \PSL_2(\R),$$
and this action extends to the boundary $\mathbf{P}^1(\R):=\R\cup\{\infty\}$ by the same formula. We equip $G$ with the Haar measure $\mu_{G}$ normalized so that the push forward of $\mu_{G}$ along $G\rightarrow G/\mathrm{PSO}_2\cong \Hb$ equals the  hyperbolic measure $\frac{dxdy}{y^2}$ in the usual coordinates $z=x+iy\in \Hb$. Let $\Gamma\leq G$ be a discrete subgroup which is cofinite with respect to $\mu_{G}$ and denote by 
$$Y_\Gamma:=\Gamma\backslash \Hb,$$ 
the hyperbolic orbifold associated to $\Gamma$. We denote the canonical projection map by
\begin{equation}\label{eq:piG}\pi_\Gamma: \Hb\rightarrow Y_\Gamma,\quad z\mapsto (z\modulo \Gamma).\end{equation} 
Denote by  $\mathrm{cusp}(\Gamma)\subset \mathbf{P}^1(\R)$ the \emph{cusps} of $\Gamma$ (i.e.\ elements of $\mathbf{P}^1(\R)$ fixed by a parabolic element of $\Gamma$) and for $\mathfrak{a}\in \mathrm{cusp}(\Gamma)$ let $\Gamma_\mathfrak{a}\subset \Gamma$ denote the stabilizer of $\mathfrak{a}$.  
For $\mathfrak{a}\in \mathrm{cusp}(\Gamma)$ we pick a \emph{scaling matrix} $\sigma_\mathfrak{a}$ for $\mathfrak{a}$, meaning any matrix which satisfies $(\sigma_\mathfrak{a})^{-1} \Gamma_\mathfrak{a} \sigma_\mathfrak{a} = \{\begin{psmallmatrix} 1 & \Z \\ & 1\end{psmallmatrix} \}$.
We denote the usual compactification of $Y_\Gamma$ by $X_\Gamma=\overline{Y_\Gamma}$ obtained by adjoining the cusps (i.e.\ $\partial Y_\Gamma=X_\Gamma\setminus Y_\Gamma$ is in one-to-one correspondence with the cosets of cusps $\Gamma\backslash \mathrm{cusp}(\Gamma) $). The quotient $Y_\Gamma$ inherits the structure of a Riemannian orbifold from $\Hb$ with length and volume elements
\begin{equation}\label{eq:hyperbolic}ds(z)=\frac{|dz|}{\Im z},\quad d\mu(z)=\frac{dxdy}{(\Im z)^2},\end{equation}
written in terms of the coordinates $z=x+iy\in \Hb$. Given a measurable subset $\BB\subset Y_\Gamma$ we denote its volume by
$$\vol(\BB):=\int_{\BB}1\, d\mu(z),$$
and for brevity we put $ \vol(\Gamma):=\vol(Y_\Gamma)$. Given a curve $\CC\subset X_\Gamma$ we define the \emph{geodesic length} as 
\begin{equation}\label{eq:geodesiclength}|\CC|:=\int_\CC 1\, ds(z), \end{equation}
and recall that for a collection $\Cc$ of closed geodesics we define the \emph{total length} as
\begin{equation*}|\Cc|:=\sum_{\CC\in \Cc}|\CC|. \end{equation*}
We define the \emph{unit tangent bundle} $\mathbf{T}^1(Y_\Gamma)$ of $Y_\Gamma$ to be the homogenous space $\Gamma\backslash G$ (ignoring any orbifold subtleties). Denote the canonical projection from the unit tangent bundle to $Y_\Gamma$ by
\begin{equation}\label{eq:pG}p_\Gamma: \Gamma\backslash G\rightarrow Y_\Gamma,\quad g\mapsto g.i\end{equation}
An oriented geodesic segment in $Y_\Gamma$ lifts canonically to the unit tangent bundle and we will (by slight abuse of notation) denote the lift by the same symbol (e.g. $\CC$). We denote the diagonal subgroup of $G$ by  
$$A:=\{a_t\in G: t\in\R \},\quad a_t:=\begin{psmallmatrix} e^{t/2}&0\\0&e^{-t/2} \end{psmallmatrix}.$$
The right action of $A$ on $\Gamma\backslash G$ generates the geodesic flow and so the lift to the unit tangent bundle of oriented closed geodesics are exactly the compact orbits of $A$. Given an oriented closed geodesic $\CC\subset \Gamma\backslash G$, we equip it with the unique $A$-invariant measure $\mu_\CC$ such that the total mass $\mu_\CC(\CC)=|\CC|$ equals the geodesic length. 
\begin{defi}\label{def:equid}
A sequence of collections of oriented closed geodesics $\Cc_1,\Cc_2, \ldots\subset \Gamma\backslash G$ is said to \emph{equidistribute in the unit tangent bundle} if for every continuous and bounded function $\Psi: \Gamma\backslash G\rightarrow \C$ it holds that 
\begin{equation}\label{eq:geoequid}
\frac{\sum_{\CC\in \Cc_n}\mu_\CC(\Psi)}{\sum_{\CC\in \Cc_n}|\CC|}\rightarrow  \frac{\int_{\Gamma\backslash G} \Psi(g)d\mu_G(g)}{\vol(\Gamma)},\quad \text{as }n\rightarrow \infty. 
\end{equation}
\end{defi}
\subsection{Automorphic forms}
Let $k\in 2\Z$ be an even integer. We define a \emph{weight $k$ automorphic form for $\Gamma$} as a smooth map $f:\Hb\rightarrow \C$ satisfying the weight $k$ automorphy relation 
\begin{equation}\label{eq:automorphy} f(\gamma z)=j_\gamma(z)^k f(z),\quad \gamma\in \Gamma,z\in \Hb,\end{equation}
where the automorphy factor is given by
$$ j_\gamma(z):=\frac{cz+d}{|cz+d|} ,\quad \gamma=\begin{pmatrix} a&b\\c&d\end{pmatrix}\in \mathrm{SL}_2(\R).$$

Denote the space of all weight $k$ automorphic forms for $\Gamma$ by $\AA(\Gamma , k)$ equipped with the \emph{Petersson inner-product}
$$\langle f,g\rangle:=\int_{Y_\Gamma}f(z)\overline{g(z)}d\mu(z),\quad f,g\in\AA(\Gamma , k), $$
giving rise to the Hilbert space
$$L^2(\Gamma,k):=\overline{\{f\in \mathcal{A}(\Gamma,k): \langle f,f\rangle<\infty\}}.$$
We will identify elements of $\mathcal{A}(\Gamma,0)$ with the corresponding functions on $\Gamma\backslash \Hb$. We say that $f\in \AA(\Gamma , k)$ is \emph{rapidly decaying} if for all cusps $\mathfrak{a}\in \mathrm{cusp}(\Gamma)$ and $A>0$ it holds that 
\begin{equation}
\label{eq:rapid} f(\sigma_\mathfrak{a} z)\ll_A (\Im z)^{-A},\quad \Im z\rightarrow \infty,
\end{equation}
considered as automatic when $\Gamma$ is co-compact.
Let
\begin{equation}\label{eq:levelraising} R_k := \frac k2 + (z-\bar z) \partial_z,\quad L_k := \frac k2 + (z-\bar z) \partial_{\overline{z}}  \end{equation}
be respectively the weight~$k$ raising and lowering operator, as defined in~\cite[eqs.~(4.3)-(4.4)]{DuFrIw02}, satisfying the following intertwining relations for any smooth map $f:\Hb\rightarrow \C$:
\begin{equation}
\label{eq:intertwine} (R_k f)(\gamma z)j_{\gamma}(z)^{-k-2}=R_k (f(\gamma z)j_\gamma(z)^{-k}),\quad (L_k f)(\gamma z)j_{\gamma}(z)^{-k+2}=L_k (f(\gamma z)j_\gamma(z)^{-k}), 
\end{equation}
and thus defining maps 
$$R_k:\AA(\Gamma , k)\rightarrow \AA(\Gamma , k+2),\quad L_k:\AA(\Gamma , k)\rightarrow \AA(\Gamma , k-2).$$ 
Given an automorphic form $f\in \mathcal{A}(\Gamma,k)$, we define the \emph{lift to the unit tangent bundle} as the map 
$$\tilde{f}:\Gamma\backslash G\rightarrow \C,\quad \PSL_2(\R)\ni g\mapsto j_g(i)^{-k} f(g.i).$$
This map identifies weight $k\in 2\Z$ automorphic forms with elements in the $\chi_{k}$-isotypic component of smooth functions on $\Gamma\backslash G$ with respect to the action of the maximal compact subgroup $K=\mathrm{PSO}_2$, where $\chi_{k}$ denotes the character $\begin{psmallmatrix} \cos\theta & \sin \theta\\ -\sin \theta& \cos \theta\end{psmallmatrix}\mapsto e^{2\pi i \theta k}$.
Note that for $k\neq 0$ the lift satisfies 
\begin{equation}\label{eq:kneq0}\int_{\Gamma\backslash G} \tilde{f}(g)d\mu_G(g)=0.\end{equation} 
We recall an alternative description for the geodesic period of weight $2$ automorphic forms.
\begin{lemma}\label{lem:weight2}
Let $\CC$ be an oriented closed geodesic on $\Gamma\backslash G$. Let $f\in \mathcal{A}(\Gamma,2)$ be a weight 2 automorphic form with lift $\tilde{f}$ to the unit tangent bundle. Then it holds that
\begin{equation}
\label{eq:periodweight2} \int_\CC f(z) \tfrac{dz}{\Im z}=i\cdot \mu_\CC(\tilde{f}).
\end{equation}
\end{lemma}
\begin{proof}
By conjugation we may assume that $\Gamma$ intersects the diagonal subgroup $A$ trivially. For $\gamma=\begin{psmallmatrix} a&b\\ c& d\end{psmallmatrix}\in \PSL_2(\R)$ hyperbolic with $c\neq 0$ we put 
\begin{align*}
\epsilon_\gamma:=\frac{c(a+d)}{|c(a+d)|}\in \{\pm 1\},\quad 
  r_\gamma:=\sqrt{(\tfrac{a+d}{2c})^2-c^{-2}},\\
   g_\gamma:=(2r_\gamma)^{-1/2}
\begin{pmatrix} \frac{a-d}{2c}+\epsilon_\gamma r_\gamma &\epsilon_\gamma\frac{a-d}{2c}-r_\gamma\\ 1 & \epsilon_\gamma\end{pmatrix}.\end{align*}
If $\CC$ corresponds to the conjugacy class of $\gamma\in \Gamma$ we have the following standard parametrization:
\begin{equation}\label{eq:parametrization} \CC=\{g_\gamma a_t : 0\leq t\leq |\CC| \}\subset \Gamma\backslash G, \end{equation}
which follows from the fact that $\gamma g_\gamma a_t=g_\gamma a_{t+|\CC|}$.
One now verifies by direct computation that
\begin{align*}
j_{g_\gamma a_t}(i)= \frac{\epsilon_\gamma+ie^t}{\sqrt{e^{2t}+1}},\quad  \Im (g_\gamma a_ti)= \frac{2r_\gamma}{e^t+e^{-t}}, \\
 \frac{d(g_\gamma a_t i)}{dt}=\frac{2r_\gamma e^t(2\epsilon_\gamma e^t+i(1-e^{2t}))}{(e^{2t}+1)^2}.
\end{align*}
This way we see that 
\begin{equation*}
\int_\CC f(z)\tfrac{dz}{\Im z}=\int_0^{|\CC|} f(g_\gamma a_t i) \frac{2r_\gamma e^t(2\epsilon_\gamma e^t+i(1-e^{2t}))}{(e^{2t}+1)^2( \frac{2r_\gamma}{e^t+e^{-t}})} dt = i\int_0^{|\CC|} j_{g_\gamma a_t}(i)^{-2}f(g_\gamma a_t i) dt.
\end{equation*}
By the definition of the lift $\tilde{f}$ and the parametrization of the closed geodesic (\ref{eq:parametrization}) we see that the right hand side above indeed equals $i\cdot \mu_\CC(\tilde{f})$.
\end{proof}

\subsubsection{Some spectral theory (compact case)}
The (weight $0$) hyperbolic Laplacian is defined by
\begin{equation}\label{eq:Delta}\Delta :=-L_{2} R_{0}=-R_{-2}L_0= -y^2\Big(\frac{\partial^2}{\partial x^2} + \frac{\partial^2}{\partial y^2}\Big), \end{equation}
which is symmetric and non-negative with respect to the Petersson innerproduct. The Laplacian $\Delta $ defines an unbounded operator on the Hilbert space $L^2(\Gamma,0)$ with domain given by bounded automorphic forms $f$ such that $\Delta f$ is also bounded. We say that $f\in \AA(\Gamma,0)$ is a \emph{ Maa{\ss} form for $\Gamma$} if $f$ is an eigenfunction for $\Delta$ and we denote the eigenvalue by $\lambda_f$. 

Now assume that $\Gamma$ is co-compact. In this case $\Delta$ has pure point spectrum and there exists an orthonormal basis $B(\Gamma,0)$ of $L^2(\Gamma,0)$ consisting of Maa{\ss} forms. Furthermore, the following spectral expansion converges in the uniform topology for $\Psi\in \mathcal{A}(\Gamma,0)$:
 \begin{align}\label{eq:spectralcompact}
\Psi= \sum_{f\in B(\Gamma,0)} \langle \Psi, f\rangle f, 
\end{align}
where the uniform convergence follows from the arguments in e.g. \cite[Sec. 4]{Iw}. We will also need the standard fact that for any (fixed) $\sigma>5/4$ the following sum converges absolutely and uniformly for $z\in Y_\Gamma$: 
 \begin{align}\label{eq:spectralbound}
\sum_{f\in B(\Gamma,0): \lambda_f>0} \frac{f(z)}{(\lambda_f)^\sigma}, 
\end{align}
which follows from the $L^\infty$-bound $\supp_{z\in Y_\Gamma} |f(z)|\ll (\lambda_f)^{1/4}$ due to Seeger--Sogge \cite[Cor.\ 2.2]{SeegerSogge} and an upper bound towards the Weyl law as in \cite[eq.\ (7.11)]{Iw}.
 Finally, we note that for any smooth function $\Psi:Y_\Gamma\rightarrow \C$ of compact support, $f\in B(\Gamma,0)$ a non-constant $L^2$-normalized Maa{\ss} form and any $N\geq 1$:
  \begin{align}\label{eq:boundIP}
\langle \Psi, f\rangle=\frac{1}{\lambda_f^N}\langle \Psi, \Delta^N f\rangle=\frac{1}{\lambda_f^N}\langle \Delta^N \Psi,  f\rangle\leq  \frac{1}{\lambda_f^N} \langle \Delta^N \Psi,  \Delta^N \Psi\rangle^{1/2}\langle f,  f\rangle^{1/2} \ll_{N,\Psi} \lambda_f^{-N}. 
\end{align}

\subsubsection{Incomplete Poincar\'{e} series} In the non-compact case it seems hard to apply the spectral theorem directly due to the existence of continuous spectrum. Instead we take advantage of the presence of a cusp as they give us access to a rich class of automorphic forms, namely \emph{Poincar\'{e} series}. After conjugating, we may assume that $\Gamma$ has cusp at infinity of width one and put $\Gamma_\infty:=\{\begin{psmallmatrix} 1 & \Z\\ 0 & 1\end{psmallmatrix}\}\leq \PSL_2(\R)$. It will suffice for us to consider only Poincar\'{e} series at infinity. 

For $m\in \Z$, $k\in 2\Z$ and $\psi:\R_{>0}\rightarrow \C$ smooth and rapidly decaying at $0$ and $\infty$ (i.e.\ $|\psi(y)|\ll_A \min(y^{A},y^{-A})$ for all $A>0$) we define the associated \emph{incomplete weight $k$ Poincar\'{e} series} as  
\begin{equation}\label{eq:Poincare}P_{m,k}(z|\psi):=\sum_{\gamma\in \Gamma_\infty\backslash \Gamma} \psi(\Im \gamma z)e(m\Re \gamma z) j_{\gamma}(z)^{-k}.\end{equation}
In the weight $0$ case we simply write $P_{m}(z|\psi):=P_{m,0}(z|\psi)$ which we consider as a function $Y_\Gamma\rightarrow \C$ in the obvious way. By the assumptions on $\psi$ the Poincar\'{e} series defines an element of $\mathcal{A}(\Gamma,k)$ rapidly decaying at all cusps of $\Gamma$. For completeness we provide a proof of this fact as well as a useful bound on the uniform norm.
\begin{lemma}
Let $\Gamma\leq \PSL_2(\R)$ be a discrete and cofinite subgroup with a cusp at infinity of width one. Let $m\in \Z$, $k\in 2\Z$ and $\psi: \R_{>0}\rightarrow \C$ a smooth function rapidly decaying at $0$ and $\infty$. Then the associated incomplete Poincar\'{e} series $P_{m,k}(z|\psi )$ is an automorphic form for $\Gamma$ of weight $k$ which is rapidly decaying at all cusps of $\Gamma$. Furthermore, if $\psi\in C_c^\infty(\R_{>0})$ has support contained in a (fixed) compact interval $I\subset \R_{>0}$ then
\begin{equation}
\label{eq:supnormPoincare} |\!| P_{m,k}(\cdot|\psi)|\!|_\infty\ll_{I,\Gamma} |\!| \psi|\!|_\infty.
\end{equation}  
\end{lemma}
\begin{proof} Let $\mathfrak{a}$ be any cusp of $\Gamma$ with scaling matrix $\sigma_\mathfrak{a}$. Rewriting as in \cite[eq.\ (3.17)]{Iw} in terms of the double coset decomposition \cite[Thm.\ 2.7]{Iw} and using that $\psi$ is rapidly decaying we see that the sum defining $P_{m,k}(\sigma_\mathfrak{a}z|\psi )$ is absolutely bounded by
\begin{align}
\nonumber 
\sum_{\gamma \in \Gamma_\infty\backslash \Gamma\sigma_\mathfrak{a}} |\psi(\Im (\gamma z))| &=\delta_{\mathfrak{a},\infty} |\psi(y)|+\sum_{\substack{c>0, d \mod c:\\ \begin{psmallmatrix}\ast & \ast \\ c& d \end{psmallmatrix}\in \Gamma\sigma_\mathfrak{a}}}\,\, \sum_{m\in \Z}\left|\psi\left(\frac{y}{|c(z+m)+d|^2}\right)\right|\\
\nonumber&\ll_A \delta_{\mathfrak{a},\infty} |\psi(y)|+\sum_{\substack{c>0, d \mod c:\\ \begin{psmallmatrix}\ast & \ast \\ c& d \end{psmallmatrix}\in \Gamma\sigma_\mathfrak{a}}}\,\, \sum_{m\in \Z}\frac{y^A}{|c(z+m)+d|^{2A}},\end{align}
where $\delta_{\mathfrak{a},\infty}=1$ if $\mathfrak{a}=\infty$ and $\delta_{\mathfrak{a},\infty}=0$ otherwise and $A\geq 1$. By using the rapid decay and bounding the sum over $m\in \Z$ by an integral over $(-\infty,\infty)$ we can bound the above by 
\begin{align}\nonumber&\ll_A y^{-A} + y^{A} \sum_{\substack{c>0, d \mod c:\\ \begin{psmallmatrix}\ast & \ast \\ c& d \end{psmallmatrix}\in \Gamma\sigma_\mathfrak{a}}} \int_{-\infty}^\infty \frac{1}{((c(x+t)+d)^2+c^2y^2)^{A}} dt\\
\nonumber&\ll_A y^{-A} + y^{-A+1} \int_{-\infty}^\infty \frac{1}{(t^2+1)^{A}} dt\sum_{\substack{c>0, d \mod c:\\ \begin{psmallmatrix}\ast & \ast \\ c& d \end{psmallmatrix}\in \Gamma\sigma_\mathfrak{a}}}c^{-2A }\\
\label{eq:poincareestimate} &\ll_A y^{-A} + y^{-A+1} \sum_{\substack{c>0, d \mod c:\\ \begin{psmallmatrix}\ast & \ast \\ c& d \end{psmallmatrix}\in \Gamma\sigma_\mathfrak{a}}}c^{-2A }, 
\end{align}
as $y\rightarrow \infty$ (corresponding to $\sigma_\mathfrak{a}z\rightarrow \mathfrak{a}$). By the following standard bound, see e.g. \cite[eq. (3.6)]{PeRi}: 
$$|\{(c,d \mod c): \begin{psmallmatrix}\ast & \ast \\ c& d \end{psmallmatrix}\in \Gamma\sigma_\mathfrak{a}, 0<c\leq T\}|\ll_\Gamma T^2,\quad T>0,$$
we conclude the wanted decay at the cusp $\mathfrak{a}$. One checks directly that the Poincar\'{e} series transforms as a weight $k$ automorphic form for $\Gamma$ as it is defined via averaging. If $\psi$ has compact support then only a finite number of terms contribute in (\ref{eq:Poincare}) yielding the bound (\ref{eq:supnormPoincare}) by the above.
\end{proof}
It is well-known that incomplete Poincar\'{e} series are dense in the Hilbert space $L^2(\Gamma,k)$ with respect to the norm topology (see e.g. \cite[Sec. 4]{LuoSarnak95}). Furthermore, it is a well established fact that if we restrict to weight $k=0$ automorphic forms of compact support
$$ C_c^\infty(Y_\Gamma):=\{f:Y_\Gamma\rightarrow \C: \text{smooth and compactly supported}\},$$
then incomplete Poincar\'{e} series are dense in the uniform topology (see \cite[Prop. 2.3]{JaasaariLesterSaha23} for a nice treatment). For our purposes we furthermore need to control the projection to the constant function $1\in L^2(\Gamma,0)$ and so for completeness we indicate the necessary changes. 
\begin{lemma}\label{lem:Poincare}
Fix a discrete and cofinite subgroup $\Gamma\leq \PSL_2(\R)$ with a cusp at $\infty$ of width $1$. Let $\Psi\in C_c^\infty(Y_\Gamma)$. Then there exists $(\psi_m)_{m\in \Z}\subset C_c^\infty(\R_{>0})$ such that 
\begin{equation}\label{eq:approxPoincare}\Psi=\sum_{m\in \Z}P_{m}(\cdot|\psi_m),\end{equation}
with the sum converging in the uniform norm. Furthermore, if $\langle \Psi,1 \rangle=0$ then $\langle P_{m}(\cdot|\psi_m),1 \rangle=0$ holds for all $m\in \Z$.
\end{lemma}
\begin{proof}
This follows from the arguments in the proof of \cite[Prop. 2.3]{JaasaariLesterSaha23} and we simply indicate the necessary changes in our setting referring to \emph{loc. cit.} for further details. Let $\BB(\Psi)\subset \Hb$ be a compact subset such that $\supp(\Psi)\subset\pi_\Gamma(\BB(\Psi))$. For each point $z\in \BB(\Psi)$ let $z\in \FF(z)\subset \Hb$ be a fundamental domain for $\Gamma$. For $z$ which is not fixed by an element of $\Gamma$, let $z\in \mathcal{U}(z)\subset \Hb$ be a domain such that $\overline{\mathcal{U}(z)}\subset \FF(z)$. Note that $\pi_\Gamma$ restrict to a diffeomorphism of $\mathcal{U}(z)$ onto its image. If $z\in \BB(\Psi)$ has a non-trivial stabilizer $\Gamma_z\subset \Gamma$, then we instead pick $ \mathcal{U}(z)\ni z$ such that $\mathcal{U}(z)=\cup_{\sigma\in \Gamma_z}\sigma( \mathcal{U}(z)\cap \overline{\FF(z)})$. By compactness we can find a finite set of points $z_1,\ldots ,z_n\in \BB(\Psi)$ such that $\BB(\Psi)\subset \cup_{i=1}^n \mathcal{U}(z_i)$. Now we pick a partition of unity of $\supp(\Psi)$ subordinate to the cover $\pi_\Gamma(\mathcal{U}(z_i))$. This way we can write 
\begin{equation}\label{eq:Psii}\Psi=\sum_{i=1}^n \Psi_i,\end{equation}
with $\Psi_i$ smooth and supported on $\pi_\Gamma(\mathcal{U}(z_i))$. Let $\tilde{\Psi}_i:\Hb\rightarrow \C$ be the smooth function of compact support which equals $\Psi_i$ on $\mathcal{U}(z_i)$ and is zero outside. Now we extend $\tilde{\Psi}_i$ to a smooth periodic function and Fourier expand using Poisson summation:
\begin{align}\label{eq:1periodic}\sum_{\sigma\in \Gamma_\infty} \tilde{\Psi}_i(\sigma z)=\sum_{m\in \Z} \psi_{i,m}(y)e(mx).\end{align}
Smoothness and compact support of $\tilde{\Psi}_i$ imply that $\psi_{i,m}$ is of compact support contained in a fixed interval and  satisfies $|\!|\psi_{i,m}|\!|_\infty \ll_A (1+|m|)^{-A} $ for all $A>0$ by partial integration (in $x$). This yields the same bound for $|\!|P_{m}(\cdot|\psi_{i,m})|\!|_\infty$ in view of  (\ref{eq:supnormPoincare}). Putting $\psi_m:=\sum_{i=1}^n \frac{1}{\# \Gamma_{z_i}}\psi_{i,m} $ for $m\in \Z$ yields what we want by summing the equality (\ref{eq:1periodic}) for $\gamma z$ over $\gamma\in \Gamma_\infty\backslash \Gamma$ for each $i=1,\ldots, n$, using the equality (\ref{eq:Psii}) and bounding the tail using the aforementioned uniform bounds. Finally, since $\langle P_{m}(\cdot|\psi_m),1 \rangle=0$ holds automatically for $m\neq 0$ we see that $\langle \Psi,1 \rangle=0$ indeed implies $\langle P_{0}(\cdot|\psi_0),1 \rangle=0$.
\end{proof}

\section{Topology of fundamental polygons}\label{sec:fundamentalpol}
In this section, we will study the geometry and topology of fundamental polygons of discrete and cofinite Fuchsian groups. This will be key in controlling the contribution from the boundary in the equidistribution question of interest. 
\subsection{Fundamental polygons}
We refer to \cite[Chapter 9]{Beardon83} for background on the material in this section. A \emph{fundamental domain} for $\Gamma$ is a domain $\mathcal{F}\subset \Hb$ such that 
\begin{enumerate}
\item For any $z\in \Hb$ we have $\gamma z\in \overline{\mathcal{F}}$ for some $\gamma\in\Gamma$.
\item\label{item:2} If $z_1,z_2\in \overline{\mathcal{F}}$ are $\Gamma$-equivalent then $z_1=z_2$ or $z_1,z_2\in \partial \mathcal{F}$.
\end{enumerate}
Here and throughout we consider the closure of $\FF$ inside $\Hb\cup \mathbf{P}^1(\R)$. We say that a fundamental domain $\mathcal{F}$ is a \emph{fundamental polygon} for $\Gamma$ if $\FF$ is furthermore (hyperbolically) convex with boundary consisting of a finite number of geodesic segments, i.e.\ $\overline{\FF}$ is a convex hyperbolic polygon. We define a \emph{side of} $\mathcal{F}$ as a non-empty subset of the shape $\gamma \overline{\mathcal{F}}\cap \overline{\mathcal{F}}$ with $ \mathrm{Id}\neq \gamma \in\Gamma$. We define a \emph{vertex of} $\mathcal{F}$ as a non-empty subset (including possibly cusps) of the shape $\gamma_1 \overline{\mathcal{F}}\cap \gamma_2 \overline{\mathcal{F}} \cap \overline{\mathcal{F}}$ with $ \mathrm{Id},\gamma_1,\gamma_2\in \Gamma$ pairwise distinct. It can be shown that a fundamental polygon $\mathcal{F}$ has an even  number of sides which are pairwise $\Gamma$-equivalent and that $\Gamma \mathcal{F}$ gives a tessellation of $\Hb$. The multiset of elements identifying sides of $\mathcal{F}$ are called the \emph{side pairing transformations} associated to $\mathcal{F}$ which we denote by $\mathrm{sp}(\Gamma, \mathcal{F})\subset \Gamma$. Notice that if $\Gamma$ contains elements of order $2$ then two of the sides of $\FF$ might be contained in the same geodesic segment (of $\partial \FF$). To distinguish the two, we define the \emph{edges} of $\FF$ as the geodesic segments constituting $\partial \FF$, i.e.\ the set:
$$\{\LL:\text{side of $\FF$}\}/\sim,$$
where $\LL\sim\LL'$ if and only if there exists a geodesic in $\Hb$ containing both $\LL$ and $\LL'$. 

Let $\mathrm{cusp}(\FF)\subset \mathrm{cusp}(\Gamma)$ denote the set of cuspidal vertices of $\FF$. Then we define the cuspidal zone at $\mathfrak{a}\in \mathrm{cusp}(\FF)$ of height $Y>0$ as
\begin{equation}
\label{eq:cuspidalzone}
\FF_\mathfrak{a}(Y) \coloneqq \{z \in \FF : \ \Im(\sigma_{\mathfrak{a}}^{-1} z) \geq Y\},
\end{equation}
where $\sigma_{\mathfrak{a}} \in \Gamma$ is a scaling matrix for $\mathfrak{a}$, so that $\sigma_{\mathfrak{a}} \infty = \mathfrak{a}$; note that $\FF_{\mathfrak{a}}(Y)$ is independent of the choice of such a scaling matrix. We define the \emph{$Y$-bulk} of $\FF$ as:
\begin{equation}
\label{eq:bulk}
\FF(Y) := \FF\setminus \left(\bigcup_{\mathfrak{a}\in \mathrm{cusp}(\FF)}\FF_{\mathfrak{a}}(Y)\right),
\end{equation}
which by standard volume estimates for horostrips (see (\ref{eq:horostrip}) below) satisfies
\begin{equation}\label{eq:bulkvol}
\vol(\FF(Y))= \vol(\Gamma)+O_\FF(Y^{-1}),\quad \text{as }Y\rightarrow \infty.
\end{equation}

\subsection{Structure of the boundary of fundamental polygons}\label{sec:topology}
As observed in \cite{HumphriesNordentoft22} when the genus is non-trivial we will have to deal with certain \emph{topological terms} coming from the boundary $\partial \FF$. We will now explain how these can be controlled by the homology of the compactification $X_\Gamma$. For a topological space $X$ we denote by $H_1(X,\Z)$ its first singular homology group. 
It will be convenient to also use the formalism of \emph{$\Delta$-complexes}, which is a slight generalization of usual simplicial complexes. We will refer to \cite[Sec. 2.1]{Hatcher02} for the required background. 

Let $\FF$ be a fundamental polygon of $\Gamma$. Then we define the \emph{boundary graph} as the image in $X_\Gamma$ of the boundary of the fundamental domain, i.e.: 
$$\Gamma\backslash \partial \FF,$$
considered as a graph with edges given by $\Gamma$-equivalence classes of sides of $\FF$ and with vertices given by $\Gamma$-equivalence classes of vertices of $\FF$, using the terminology  of the previous section. Then we have the following key lemma concerning the topology of boundary graphs. 
\begin{lemma}\label{lem:structure}
Let $\FF$ be a fundamental polygon for discrete and cofinite subgroup $\Gamma\leq \PSL_2(\R)$. Then the boundary graph $\Gamma\backslash \partial \FF$ is homotopic to a wedge of $2g$ circles, where $g$ is the genus of $X_\Gamma$. Furthermore, the inclusion $\Gamma\backslash \partial \FF\subset X_\Gamma$ induces an isomorphism on singular homology
\begin{equation}H_1(\Gamma\backslash \partial\FF,\Z)\cong H_1(X_\Gamma,\Z)\cong \Z^{2g}. \end{equation}
\end{lemma}
\begin{proof}
Let $z_0\in \Hb$ be a point in the interior of $\FF$ and let $x_0:=\pi_\Gamma(z_0)\in Y_\Gamma\subset X_\Gamma$ be its projection modulo $\Gamma$. Since $\overline{\FF}$ is a polygon (and thus homeomorphic to a ball) there exists a deformation retract of $\overline{\FF}\setminus\{z_0\}$ onto  the boundary $\partial\FF$ which projects to a deformation retract of $X_\Gamma\setminus\{x_0\}$ onto $\Gamma\backslash \partial\FF$. Now $X_\Gamma\setminus\{x_0\}$ is a genus $g$ curve with a puncture and is thus homotopic to a wedge of $2g$ circles which yields the first claim. Consider now the maps on singular homology induced by the natural inclusions: 
$$H_1(\Gamma\backslash \partial \FF,\Z)\rightarrow  H_1( X_\Gamma\setminus \{x_0\},\Z) \rightarrow H_1(X_\Gamma,\Z). $$
The first map is an isomorphism by the above deformation retract. The second map is an isomorphism by a standard argument using the Mayer--Vietoris sequence. Finally, since $X_\Gamma$ is a closed surface of genus $g$ we have $H_1(X_\Gamma,\Z)\cong \Z^{2g}$. 
\end{proof} 
By choosing an orientation on each edge of the boundary graph, we realize $\Gamma\backslash \partial \FF$ as a $1$-dimensional $\Delta$-complex with $0$-simplices given by the vertices of $\Gamma\backslash \partial \FF$ and the $1$-simplices given by the oriented edges of $\Gamma\backslash \partial \FF$. Denote by 
$$\Delta_1(\Gamma\backslash \partial\FF):=\Z[\text{$1$-simplices of }\Gamma\backslash \partial\FF],$$ 
the simplicial $1$-chains of the boundary graph, and by 
$$\Delta^\mathrm{cl}_1(\Gamma\backslash \partial\FF)\subset \Delta_1(\Gamma\backslash \partial\FF),$$
the subgroup of closed simplicial $1$-chains, i.e.\ the kernel of the boundary map given by sending a $1$-chain to the (oriented) difference between its endpoints: 
$$\partial_1:\Delta_1(\Gamma\backslash \partial\FF)\rightarrow \Delta_0(\Gamma\backslash \partial\FF):=\Z[\text{$0$-simplices of }\Gamma\backslash \partial\FF].$$
Since $\Gamma\backslash \partial\FF$ is one dimensional, the $\Delta$-simplicial homology group is given by the following:
$$ H^\Delta_1(\Gamma\backslash \partial\FF):=\ker \partial_1/\mathrm{im}\, \partial_2=\Delta^\mathrm{cl}_1(\Gamma\backslash \partial\FF), $$
which is isomorphic to the singular homology group $H_1(\Gamma\backslash \partial\FF,\Z)$ by general principles (see \cite[Sec. 2.1]{Hatcher02}). In particular, by Lemma \ref{lem:structure} the inclusion $\Gamma\backslash \partial \FF\subset X_\Gamma$ induces an isomorphism on homology
\begin{equation}\label{eq:isohom}H^\Delta_1(\Gamma\backslash \partial\FF)\cong  H_1(X_\Gamma,\Z),\end{equation}
given by viewing (closed) simplicial $1$-chains as (closed) singular $1$-chains on $X_\Gamma$ via the natural inclusion. Fix an integral basis 
\begin{equation}\label{eq:bnd}\CC_{1}(\FF),\ldots ,\CC_{2g}(\FF)\in \Delta^\mathrm{cl}_1(\Gamma\backslash \partial\FF)\cong \Z^{2g},\end{equation} 
and denote by $[\CC_{i}(\FF)]\in H_1(X_\Gamma,\Z)$ the image of  $\CC_{i}(\FF)$ under the isomorphism (\ref{eq:isohom}). It follows that  
$$[\CC_{1}(\FF)],\ldots ,[\CC_{2g}(\FF)]\in H_1(X_\Gamma,\Z)$$ 
form an integral basis for $H_1(X_\Gamma,\Z)$. 
Let 
\begin{equation}\label{eq:coho}\omega_{1}(\FF),\ldots ,\omega_{2g}(\FF)\in H^1(X_\Gamma,\Z),\end{equation}
be the dual basis in cohomology with respect to the integral Poincar\'{e} pairing
$$\langle\cdot, \cdot\rangle_\mathrm{P}: H_1(X_\Gamma,\Z)\times H^1(X_\Gamma,\Z)\rightarrow \Z. $$ 
It is a standard fact of Hodge theory that for each $i=1,\ldots, 2g$ we can represent  $\omega_{i}(\FF)$ (under the de Rham isomorphism) by a differential $1$-form of the shape
\begin{equation}\label{eq:DF}  f_{i}(z)dz+\overline{g_i(z)dz},\end{equation}
where $f_i,g_i: \Hb\rightarrow \C$ are \emph{cuspidal, holomorphic modular forms of weight $2$ for $\Gamma$}, i.e.\ holomorphic maps $f_i:\Hb\rightarrow \C$ vanishing at the cusps of $\Gamma$ and satisfying 
$$f_i(\gamma z)=(cz+d)^2 f_i(z), \quad z\in \Hb, \gamma\in \Gamma,$$
and similarly  for $g_i$.  Notice that if $f$ is a holomorphic modular form of weight $2$ for $\Gamma$ then $z\mapsto (\Im z) f(z)$ defines an automorphic form of weight $2$ for $\Gamma$ in the sense of equation (\ref{eq:automorphy}). 

\section{Partial coverings defined from geodesics ({\lq}{\lq}painting to the left{\rq{\rq}})}\label{sec:PC}
In this section we will give the formal definition of the notion of \emph{partial covering} alluded to in the introduction, then obtain examples via {\lq}{\lq}painting to the left{\rq{\rq}} and finally, compare the partial coverings to the constructions in \cite{DuImTo} and \cite{HumphriesNordentoft22}. We encourage the geometrically oriented reader to consult the figures in Example \ref{ex:triangle}. 

As above, let $\Gamma\leq \PSL_2(\R)$ be a discrete and cofinite subgroup and let $\FF$ be a fundamental polygon for $\Gamma$. Denote by $\Hom_\mathrm{meas}(Y_\Gamma,\Z_{\geq 0})$ the monoid of measurable maps $Y_\Gamma\rightarrow \Z_{\geq 0}$ with addition defined on the target. We define the \emph{(monoid of) partial coverings of $Y_\Gamma$} as 
$$\mathrm{Partial}(Y_\Gamma):=\Hom_\mathrm{meas}(Y_\Gamma,\Z_{\geq 0})/\langle\varphi\in \Hom_\mathrm{meas}(Y_\Gamma,\Z_{\geq 0}):\supp (\varphi)\text{ null set}\rangle. $$ 
For any measurable finite volume subset $\mathcal{B}\subset \Hb$ we obtain a partial covering which (by slight abuse of notation) we denote by the same symbol:
\begin{equation}\label{eq:BBsubset}\mathcal{B}:[x\mapsto \#(\pi_\Gamma^{-1}(\{x\})\cap \mathcal{B})]\in\mathrm{Partial}(Y_\Gamma),\end{equation}
where $\pi_\Gamma:\Hb\rightarrow Y_\Gamma$ denotes the canonical projection map (\ref{eq:piG}). Note that this is well-defined since 
$$\{x\in Y_\Gamma:\#(\pi_\Gamma^{-1}(\{x\})\cap \mathcal{B})=\infty\},$$
has measure zero due to the fact that $\mathcal{B}$ has finite volume. 

Let $\CC$ be an oriented closed geodesic on $Y_\Gamma$ and let $\tilde{\CC}\subset \Hb$ be a lift to the universal cover. Let $\mathcal{S}\subset \Hb$ be the unique infinite (two-sided) oriented geodesic  containing $\tilde{\CC}$. Let $x_0,x_1\in\mathbf{P}^1(\R)$ be  the ordered endpoints of $\mathcal{S}$. Let $\sigma\in \PSL_2(\R)$ be a matrix such that 
$$ \sigma \mathcal{S}= i\R,\quad \sigma x_0=\infty,\quad  \sigma x_1=0,$$
(note that this condition defines a one-parameter subgroup of matrices $\sigma$) and put 
\begin{equation}\label{eq:tildeB}\mathcal{B}:=\sigma^{-1}\{z\in \Hb: \Re z\geq 0\}\subset \Hb, \end{equation}   
i.e.\ the interior of $\mathcal{S}$ relative to the orientation (which does not depend on the choice of $\sigma$). Let $\{\gamma_1,\ldots, \gamma_m\}\subset \Gamma$ be the set of matrices (possibly empty) such that $\tilde{\CC}$ intersects $\gamma_i\FF$.  The subsegment of $\tilde{\CC}$ contained in $\overline{\gamma_i\FF}$  together with the part of the boundary $\partial(\gamma_i\FF)$ connecting the (oriented) intersection $\partial(\gamma_i\FF)\cap \tilde{\CC}$ bounds a hyperbolic polygon, namely:
\begin{equation}\label{eq:PPi} \PP_i:=\mathcal{B}\cap \gamma_i\FF.\end{equation}
This is depicted in Figures \ref{fig:1} and \ref{fig:2}.
\begin{defi}
We define the \emph{the partial covering of $\FF$ with respect to $\CC$ and $\FF$} as 
\begin{equation}\label{eq:PPdef}\PP(\CC,\FF):=\sum_{i=1}^m \PP_i\in \mathrm{Partial}(Y_\Gamma),\end{equation} 
with $\PP_i$ given by eq. (\ref{eq:PPi}) considered as an element of $\mathrm{Partial}(Y_\Gamma)$ using eq. (\ref{eq:BBsubset}) above. Here we put $\PP(\CC,\FF)=0$ if $\CC$ is contained in $\Gamma\backslash\partial{\FF}$. Note that this definition does not depend on the choice of the lift $\tilde{\CC}$. 
\end{defi}
Note that $Y_\Gamma$ can be partitioned into a finite number of hyperbolic polygons such that $\PP(\CC,\FF)$, thought of as a function $Y_\Gamma\rightarrow \Z_{\geq 0}$, is constant when restricted to the interior of each polygon. This is depicted in Figure \ref{fig:3} for the $(2,4,6)$ triangle group. 

Similarly if $\mathbf{C}=\{\CC_1,\ldots, \CC_N\}$ is a finite collection of oriented closed geodesics, we define  
$$\PP(\Cc,\FF):=\sum_{\CC\in \Cc}\PP(\CC,\FF),\qquad |\Cc|:=\sum_{\CC\in \Cc}|\CC|.$$
We define the \emph{volume} of $\PP\in \mathrm{Partial}(Y_\Gamma)$ as: 
$$\vol(\PP):=\int_{Y_\Gamma} \PP(z) d\mu(z)\in [0,\infty]. $$
For an integrable function $\phi:Y_\Gamma\rightarrow \C$ we define the \emph{integral of $\phi$ over $\PP$} as: 
$$\int_{\PP} \phi:=\int_{Y_\Gamma} \PP(z)\phi(z) d\mu(z), $$
and for a measurable subset $\mathcal{B}\subset Y_\Gamma$ we denote
\begin{equation}\label{eq:volumeintersection} \vol(\PP\cap \mathcal{B}):= \int_{\PP} \mathbf{1}_\mathcal{B}. \end{equation}
In the setting of \cite{DuImTo} it was observed that the partial coverings associated to an oriented closed geodesics and its opposite covers the modular surface evenly. This is also true in our setting.
\begin{lemma}\label{lem:trivial}
Let $\CC$ be an oriented closed geodesic on $Y_\Gamma$ and let $\overline{\CC}$ denote the oriented closed geodesic with the same image and opposite orientation. Then 
$$\PP(\CC,\FF)+\PP(\overline{\CC},\FF)=nY_\Gamma\in \mathrm{Partial}(Y_\Gamma),$$
for some $n\in \Z_{\geq 0}$.  
\end{lemma}
\begin{proof}
Let $\tilde{\CC}\subset \Hb$ be a lift of $\CC$. Then the opposite of $\tilde{\CC}\subset \Hb$ is a lift of $\overline{\CC}$. Let $\mathcal{B}$ and $\mathcal{B}'$ be the interior of the infinite geodesics containing respectively, $\tilde{\CC}$ and its opposite, as defined in (\ref{eq:tildeB}). Then we see that
$$\mathcal{B}\cup\mathcal{B}'=\Hb,\quad \mu(\mathcal{B}\cap\mathcal{B}')=0,$$
and so the corresponding polygons $\PP_i$ and $\PP_i'$ (if any such exists) tile the corresponding $\Gamma$-translate of $\FF$. This yields the wanted.
\end{proof}
We will now consider the above constructions in a specific co-compact example. 
\begin{example}\label{ex:triangle}
Consider the $(2,4,6)$ triangle group, i.e.\ the discrete and co-compact subgroup
$$\Gamma_\triangle:=\left\langle \begin{psmallmatrix}0 & -1\\ 1& 0 \end{psmallmatrix},\begin{psmallmatrix}1/\sqrt{2} & (1+\sqrt{3})/2\\ (1-\sqrt{3})/2& 1/\sqrt{2} \end{psmallmatrix}\right\rangle=\langle S, \sigma\mid  S^2=\sigma^4=(\sigma S)^6=1\rangle\subset \PSL_2(\R).$$
The corresponding hyperbolic orbifold has signature $(0; 2,4,6;0)$ and a fundamental polygon $\FF_\triangle\subset \Hb$ is given by the (hyperbolic) triangle with vertices $\{\tfrac{1+i}{\sqrt{2}},\tfrac{-1+i}{\sqrt{2}}, \tfrac{i(1+\sqrt{3})}{\sqrt{2}}\}$ (see Figure \ref{fig:3}) for which $\mathrm{sp}(\Gamma_\triangle, \FF_\triangle)=\{S,\sigma, \sigma^{-1}\}$ are the side pairing transformations. Consider the following three hyperbolic elements of $\Gamma_\triangle$: 
\begin{equation} \gamma_1=\sigma^2 S=\begin{psmallmatrix}\sqrt{2}/(1+\sqrt{3}) & 0\\ 0&(1+\sqrt{3})/\sqrt{2} \end{psmallmatrix},\quad  \gamma_2=\sigma S \sigma^3 S=\begin{psmallmatrix}(3+\sqrt{3})/2 & 1/\sqrt{2}  \\ 1/\sqrt{2} & \sqrt{3}/(1+\sqrt{3}) \end{psmallmatrix},\end{equation} \begin{equation}  \gamma_3=S\sigma^2 S\sigma^2S\sigma =\begin{psmallmatrix} -2/(\sqrt{3} + 1)^3 & \sqrt{2}/(\sqrt{3} + 1)^2 \\ -(\sqrt{3} + 1)^2/(2\sqrt{2}) & -(\sqrt{3} + 1)^3/4 \end{psmallmatrix} \end{equation}
whose conjugacy classes  correspond to three oriented closed geodesics $\CC_1,\CC_2,\CC_3$.  Figures \ref{fig:1} and \ref{fig:2} show a specific  choice of lifts $\tilde{\CC}_1,\tilde{\CC}_2,\tilde{\CC}_3\subset \Hb$ such that $\tilde{\CC}_i$ is contained in the axis of $\gamma_i$ (i.e.\ the unique infinite geodesic stabilized by $\gamma_i$). Furthermore, Figures \ref{fig:1} and \ref{fig:2} depict the corresponding polygons $\PP_1,\ldots, \PP_m$ obtained by considering the intersection of $\tilde{\CC}_1,\tilde{\CC}_2,\tilde{\CC}_3$ with the tessellation of $\Hb$ by the fundamental polygon $\FF_\triangle$. Note that $\CC_1$ and $\CC_2$ are both \emph{reciprocal geodesics} for $\Gamma_\triangle$ in the terminology of \cite{SarnakReciprocal} meaning that $\gamma_1$ and $\gamma_2$ are both conjugate in $\Gamma_\triangle$ to their respective inverses. This means geometrically that $\CC_1=\overline{\CC}_1$ and $\CC_2=\overline{\CC}_2$ (i.e.\ the closed geodesics retrace back over themselves) which in view of Lemma \ref{lem:trivial} implies that the corresponding partial coverings $\mathcal{P}(\mathcal{C}_i,\mathcal{F}_\triangle),i=1,2$ are constant functions, i.e.\ they cover $Y_{\Gamma_\triangle}$ perfectly. This is not the case for $\CC_3$ and the partial covering associated to $(\CC_3,\FF_\triangle)$ is depicted in Figure \ref{fig:3} with gradient corresponding to the values of the function $\mathcal{P}(\mathcal{C}_3,\mathcal{F}_\triangle): Y_{\Gamma_\triangle}\rightarrow \Z_{\geq 0}$ (from lightest $\leftrightarrow 3$ to darkest $\leftrightarrow 8$).  

\tikzmath{
\r=2.7321;
\d=1.9318;
\c=0; 
\cc1=\d; 
\rr1=\r;
\ccc2=-\d; 
\rrr2=\r;
\cc2=0.5176380893; 
\rr2=0.7320508067;
\cc3=-0.5176380913; 
\rr3=0.7320508087;
\cc4=-1.931851653; 
\rr4=2.732050807;
\cc5=-3.346065213; 
\rr5=2.732050807;
\cc6=-0.8965754733; 
\rr6=0.7320508087;
\cc7=3.346065213; 
\rr7=2.732050807;
\cc8=-0.3789373823;
\rr8=0.2679491925;
\cc9=-1.414213561;
\rr9=1; 
\rrr1=0.2679491926; 
\ccc1=0;
\ccc2=-0.1387007082;
\rrr2=0.1961524226; 
\ccc3=0.1387007082;
\rrr3=0.1961524226; 
\ccc4=-0.2402366743; 
\rrr4=0.1961524231;
\ccc5=0;
\rrr5=0.07179676984;
\ccc6=0;
\rrr6=0;
\ccc7=0;
\rrr7=0;
\ccc8=0;
\rrr8=0;
\ccccc2=13.19479210;
\rrrrr2=12.65599998;
\ccccc4=0.2869633204;
\rrrrr4=1.376227736;
\ccccc6=3.535533899;
\rrrrr6=3.391164988;
\ccccc8=0.9473434545;
\rrrrr8=0.9086599200;
\ccccc1=1.224744872;
\rrrrr1=1.581138830;
\ccccc3=-1.224744872;
\rrrrr3=1.581138830;
\ccccc5=-1.224744872;
\rrrrr5=1.581138830;
\ccccc7=0;
\rrrrr7=0;
\ccccc9=0;
\rrrrr9=0;
\geodesicc=-0.9473434545;
\geodesicr=0.9086599200;
\sl1=-1/(\d+\r);
\sr1=-1/(\d-\r);
\c1=(\sl1+\sr1)/2;
\rrrrrr1=(\sl1-\sr1)/2;
\sl2=-1/(-\d+\r);
\sr2=-1/(-\d-\r);
\c2=(\sl2+\sr2)/2;
\rrrrrr2=(\sl2-\sr2)/2;
\sl3=(\d+1)/(\d^2-\r^2+\d);
\sr3=(-\d+1)/(-\d^2+\r^2+\d);
\c3=(\sl3+\sr3)/2;
\rrrrrr3=(\sl3-\sr3)/2;
\sl4=(-5.664326539)/(-3.663871272);
\sr4=(1.800312672)/(0.1998320314);
\c4=(\sl4+\sr4)/2;
\rrrrrr4=(\sr4-\sl4)/2;
\sl5=(-0.3861245733)/(-0.9318696580);
\sr5=(7.078514637)/(2.931833645);
\c5=(\sl5+\sr5)/2;
\rrrrrr5=(\sr5-\sl5)/2;
\sl6=(2.931833646)/(7.078514639);
\sr6=(-0.9318696582)/(-0.3861245734);
\c6=(\sl6+\sr6)/2;
\rrrrrr6=(\sr6-\sl6)/2;
\sl7=(0.1998320314)/(1.800312672);
\sr7=(-3.663871273)/(-5.664326540);
\c7=(\sl7+\sr7)/2;
\rrrrrr7=(\sr7-\sl7)/2;
\sl8=(0.8587581124)/(0.1414746763);
\sr8=(1.590782558)/(2.591015347);
\c8=(\sl8+\sr8)/2;
\rrrrrr8=(\sr8-\sl8)/2;
\sl9=(1.346132660)/(2.318105649);
\sr9=(-5.346257405)/(-5.147330343);
\c9=(\sl9+\sr9)/2;
\rrrrrr9=(\sr9-\sl9)/2;
\ssl3=2.805883700;
\ssr3=-0.3563939582;
\ccccccc3=(\ssl3+\ssr3)/2;
\rrrrrrr3=(\ssr3-\ssl3)/2;
}

\begin{figure}
\begin{tikzpicture}[scale=2.5]
\clip (-1.1,-0.2) rectangle (1.1,2);
\fill[white] (-1.1,-0.3) rectangle (1.1,2);
\begin{scope}
  \clip (-\d,0) circle(\r);
   \clip (\d,0) circle(\r);
  \fill[black,opacity=0.2] (0,0) rectangle (1,4);
  \draw[thick] (0,0)--(0,2);
  \draw[very thin] (0,0) circle (1);
  \draw[very thin]  (\c1,0) circle(\rrrrrr1);
\draw[very thin]   (\c2,0) circle(\rrrrrr2);
\end{scope}
\begin{scope}
\clip (-\d,0) circle(\r);
\draw[ultra thin] (\d,0) circle(\r);
\end{scope}
\begin{scope}
\clip (\d,0) circle(\r);
\draw[ultra thin] (-\d,0) circle(\r);
\end{scope}
\fill[white]  (\c1,0) circle(\rrrrrr1);
\fill[white]  (\c2,0) circle(\rrrrrr2);
\draw[ultra thin,decoration={markings,
      mark=between positions 0 and 1 step 14pt with {\arrowreversed{to}}},
    postaction={decorate}] (0,0) -- (0,2);
  \fill[white] (-1,-0.3) rectangle (1,0);
  \draw[line width=0mm,gray] (-2,0) -- (2,0);
 
  \node at (0,0) [below] {{\tiny\color{gray} $ 0$}};
    \node at (-1,0) [below] {{\tiny\color{gray} $ -1$}}; 
      \node at (1,0) [below] {{\tiny\color{gray} $ 1$}};
      \draw[line width=0mm, gray] (0,0) -- (0,0.03);
    \draw[line width=0mm, gray] (-1,0) -- (-1,0.03);
      \draw[line width=0mm, gray] (1,0) -- (1,0.03);
      \node at (0.3,1.2) {$\mathcal{P}_1$}; 
        \node at (0.2,0.8) {$\mathcal{P}_2$};
\end{tikzpicture}
\begin{tikzpicture}[scale=2.5]
\clip (-1.1,-0.2) rectangle (1.1,2);
\fill[white] (-1.1,-0.2) rectangle (1.1,2);
\begin{scope}
  \clip (-\d,0) circle(\r);
  \fill[black,opacity=0.2] (\d,0) circle(\r);
\end{scope}
\fill[white] (\ccccccc3,0) circle(\rrrrrrr3);
\fill[white]  (\c1,0) circle(\rrrrrr1);
\fill[white] (-\c7,0) circle(\rrrrrr7);

\begin{scope}
  \clip (-\d,0) circle(\r);
  \draw[very thin] (\d,0) circle(\r);
\end{scope}

\begin{scope}
  \clip (\d,0) circle(\r);
  \draw[very thin] (-\d,0) circle(\r);
\end{scope}
\draw[very thin] (0,0) circle(1);
\begin{scope}
  \clip (-\d,0) circle(\r);
  \draw[very thin]  (\c1,0) circle(\rrrrrr1);
\end{scope}
\draw[very thin] (\c2,0) circle(\rrrrrr2);
\draw[very thin](\c3,0) circle(\rrrrrr3);
\begin{scope}
  \clip (-\d,0) circle(\r);
  \draw[thick] (\ccccccc3,0) circle(\rrrrrrr3);
\end{scope}
\draw[very thin] (-\c7,0) circle(\rrrrrr7);
\begin{scope}
  \clip (\d,0) circle(\r);
  \draw[very thin] (-\c6,0) circle(\rrrrrr6);
\end{scope} 
\fill[white] (-\c7,0) circle(\rrrrrr7); 
\fill[white] (-\c6,0) circle(\rrrrrr6);
\fill[white] (\c1,0) circle(\rrrrrr1);
  \draw[line width=0mm, decoration={markings,
      mark=between positions 0 and 1 step 14pt with {\arrowreversed{to}}},
    postaction={decorate}] (\ccccccc3,0) circle(\rrrrrrr3);
  \fill[white] (-2,-0.3) rectangle (2,0);
   \draw[line width=0mm,gray] (-2,0) -- (2,0);
 
  \node at (0,0) [below] {{\tiny\color{gray} $ 0$}};
    \node at (-1,0) [below] {{\tiny\color{gray} $ -1$}}; 
      \node at (1,0) [below] {{\tiny\color{gray} $ 1$}};
      \draw[line width=0mm, gray] (0,0) -- (0,0.03);
    \draw[line width=0mm, gray] (-1,0) -- (-1,0.03);
      \draw[line width=0mm, gray] (1,0) -- (1,0.03);
      \node at (0,1.4) {$\mathcal{P}_4$}; 
        \node at (-0.3,0.8) {$\mathcal{P}_3$};
        \node at (-0.35,0.6) {\tiny $\mathcal{P}_2$};
        \node at (-0.4,0.4) {\tiny $\mathcal{P}_1$};
\end{tikzpicture}
\caption{Polygons $\PP_1,\ldots, \PP_m$ (shaded) obtained from the closed geodesics $\CC_1$ and $\CC_2$ when $\Gamma=\Gamma_\triangle$ and $\mathcal{F}=\mathcal{F}_\triangle$.}\label{fig:1}
\end{figure}
\begin{figure}

\begin{tikzpicture}[scale=3]
\clip (-2.8,-0.2) rectangle (1.1,3);
\fill[white] (-3,-0.7) rectangle (3,4);
\fill[black,opacity=0.2] (-\d,0) circle(\r);
\draw[ultra thin] (\cc4,0) circle (\rr4);
\fill[white] (\cc2,0) circle (\rr2-0.0006); 

\begin{scope}
  \clip (\cc2,0) circle (\rr2);
  \fill[black,opacity=0.2] (\cc3,0) circle (\rr3-0.0006);
\end{scope}
\fill[white] (\geodesicc,0) circle(\geodesicr-0.0006);
\draw[ultra thin] (\ccc4,0) circle(\rrr4);
\draw[ultra thin] (\ccc5,0) circle(\rrr5);
\begin{scope}
  \clip (-\d,0) circle (\r);
  \draw[ultra thin] (0,0) circle (1); 
\end{scope}
\begin{scope}
  \clip (-\d,0) circle (\r);
  \draw[ultra thin] (\cc1,0) circle (\rr1);
\end{scope}
\draw[ultra thin] (\cc3,0) circle (\rr3);
\begin{scope}
  \clip (-\d,0) circle (\r);
  \draw[ultra thin] (\cc2,0) circle (\rr2);
\end{scope}
\draw[ultra thin] (\cc6,0) circle (\rr6);
\fill[white](\cc5,0) circle (\rr5-0.0006); 
\begin{scope}
  \clip (-\d,0) circle (\r);
  \draw[ultra thin] (\cc5,0) circle (\rr5); 
\end{scope}
\fill[white](\cc6,0) circle (\rr6-0.0006);
\draw[ultra thin] (\ccc1,0) circle (\rrr1);
\begin{scope}
  \clip (-\ccc1,0) circle (\rrr1);
 \draw[ultra thin] (\ccc3,0) circle (\rrr3);
\end{scope}
\fill[white](\ccc4,0) circle (\rrr4-0.0006);
\fill[white](\ccc5,0) circle (\rrr5-0.0006);
\fill[white](\ccc3,0) circle (\rrr3-0.0006);
\draw[line width=0.09mm] (\geodesicc,0) circle(\geodesicr);
\fill[white](\cc5,0) circle (\rr5-0.0006);

\begin{scope}
  \clip (-\ccc1,0) circle (\rrr1);
 \draw[ultra thin] (\ccc2,0) circle (\rrr2);
\end{scope}
\begin{scope}
\clip (\ccc2,0) circle (\rrr2);
\fill[black, opacity=0.2] (\ccc3,0) circle (\rrr3-0.0006);
\end{scope}
\begin{scope}
\clip (\ccc3,0) circle (\rrr3-0.0006);
\fill[white,] (\geodesicc,0) circle (\geodesicr-0.0006);
\end{scope}

\fill[white](\ccc4,0) circle (\rrr4-0.0006);
\fill[white](\ccc5,0) circle (\rrr5-0.0006);
\draw[line width=0mm, decoration={markings,
      mark=between positions 0 and 1 step 14pt with {\arrowreversed{to}}},
    postaction={decorate}] (\geodesicc,0) circle(\geodesicr);
\begin{scope}
\clip (\ccc2,0) circle (\rrr2);
\clip (\ccc3,0) circle (\rrr3);
\draw[ultra thin] (\ccc5,0) circle (\rrr5);
\end{scope}
  \fill[white] (-3,-0.3) rectangle (3,0);
  \draw[line width=0mm, gray] (-3,0) -- (3,0);
  \draw[line width=0mm, gray] (0,0) -- (0,0.03);
    \draw[line width=0mm, gray] (-1,0) -- (-1,0.03);
      \draw[line width=0mm, gray] (1,0) -- (1,0.03);
        \draw[line width=0mm, gray] (-2,0) -- (-2,0.03);
  \node at (0,0) [below] {{\tiny\color{gray} $ 0$}};
    \node at (-1,0) [below] {{\tiny\color{gray} $ -1$}}; 
      \node at (-2,0) [below] {{\tiny\color{gray} $ -2$}}; 
      \node at (1,0) [below] {{\tiny\color{gray} $ 1$}};
        \node at (-1,2) {$\mathcal{P}_1$};
        \node at (0,1.2) {$\mathcal{P}_2$}; 
        \node at (0,0.8) {$\mathcal{P}_3$};
        \node at (-0.12,0.53) {\tiny $\mathcal{P}_4$}; 
        \node at (0,0.35) {\tiny $\mathcal{P}_5$};
        \node at (0,0.2) {\scalebox{.5}{$\mathcal{P}_6$}};
        \node at (-0.03,0.14) {\scalebox{.26}{$\mathcal{P}_7$}};
        \node at (0,0.1) {\scalebox{.3}{$\mathcal{P}_8$}};
       
\end{tikzpicture}

\caption{Polygons (shaded) obtained from the closed geodesic $\CC_3$ when $\Gamma=\Gamma_\triangle$ and $\mathcal{F}=\mathcal{F}_\triangle$.}\label{fig:2}
\end{figure}
\begin{figure}
\begin{tikzpicture}[scale=3]
\clip (-1.1,-0.2) rectangle (1.1,2);
\fill[white] (-3,-1) rectangle (3,3);

 \begin{scope}
      \clip (-\d,0) circle(\r);
      \clip (\d,0) circle(\r);
      \draw[ultra thin] (\d,0) circle(\r);
          \draw[ultra thin] (-\d,0) circle(\r);
      \end{scope}
      \begin{scope}
      \clip (\d,0) circle(\r);
      \draw[ultra thin] (-\d,0) circle(\r);
      \fill[black, opacity=0.85] (-\d,0) circle(\r);
      \end{scope}

      \begin{scope}
      \clip (\d,0) circle(\r);
       \clip (-\d,0) circle(\r);
      \fill[white , opacity=0.4] (\ccccc2,0) circle(\rrrrr2);
      \fill[white , opacity=0.25] (\ccccc4,0) circle(\rrrrr4);
      \fill[white , opacity=0.3] (\ccccc6,0) circle(\rrrrr6);
      \fill[white , opacity=0.5] (\ccccc8,0) circle(\rrrrr8);
      \fill[white , opacity=0.4] (-\ccccc2,0) circle(\rrrrr2); 
      \fill[white , opacity=0.25] (-\ccccc4,0) circle(\rrrrr4);
      \fill[white , opacity=0.3] (-\ccccc6,0) circle(\rrrrr6);
      \fill[white , opacity=0.5] (-\ccccc8,0) circle(\rrrrr8); 
      
      \draw[] (\ccccc2,0) circle(\rrrrr2);
      \draw[] (\ccccc4,0) circle(\rrrrr4);
      \draw[] (\ccccc6,0) circle(\rrrrr6);
      \draw[] (\ccccc8,0) circle(\rrrrr8);
      \draw[] (-\ccccc2,0) circle(\rrrrr2);
      \draw[] (-\ccccc4,0) circle(\rrrrr4);
      \draw[] (-\ccccc6,0) circle(\rrrrr6);
      \draw[] (-\ccccc8,0) circle(\rrrrr8); 
      
      \end{scope}
     \begin{scope}
      \clip (-\d,0) circle(\r);
\draw[ultra thin] (\d,0) circle(\r);
      \end{scope}
        \begin{scope}
      \clip (\d,0) circle(\r);
      
\draw[ultra thin] (-\d,0) circle(\r);

      \end{scope}
\fill[white] (0,0) circle(1);
  \begin{scope}
      \clip (\d,0) circle(\r);
      \clip (-\d,0) circle(\r);
      \draw[ultra thin] (0,0) circle(1);

      \end{scope}
   \draw[line width=0mm,gray] (-2,0) -- (2,0);
 
  \node at (0,0) [below] {{\tiny\color{gray} $ 0$}};
    \node at (-1,0) [below] {{\tiny\color{gray} $ -1$}}; 
      \node at (1,0) [below] {{\tiny\color{gray} $ 1$}};
      \draw[line width=0mm, gray] (0,0) -- (0,0.03);
    \draw[line width=0mm, gray] (-1,0) -- (-1,0.03);
      \draw[line width=0mm, gray] (1,0) -- (1,0.03);
\end{tikzpicture}
\caption{Partial covering obtained from the pair  $(\mathcal{C}_3,\mathcal{F}_\triangle)$.}\label{fig:3}
\end{figure}  
\end{example}
 \subsection{Comparing with the construction of Duke--Imamo\={g}lu--T\'{o}th}\label{sec:comparing}  Now we restrict to $\Gamma=\PSL_2(\Z)$ and denote by $Y_0(1):=\PSL_2(\Z)\backslash \Hb$ the modular surface. In this setting, Duke--Imamo\={g}lu--T\'{o}th \cite{DuImTo} associated to an oriented closed geodesic $\CC_A$ with $A\in \Cl_K^+$, the narrow class group of a real quadratic field, a partial covering of $Y_0(1)$ with boundary given by the corresponding closed geodesic as in \cite{Sarnak82}. This was obtained by associating to $A$ a certain thin subgroup of $\Gamma_A\leq \Gamma$ with convex core (or Nielsen region) $\NN_A\subset \Hb$ and then considering the projection of $\Gamma_A\backslash \NN_A$ to $Y_0(1)$ (counted with multiplicity), which in our terminology corresponds to
\begin{equation}\label{eq:FFA}\FF_A \in \mathrm{Partial}(Y_0(1)),\end{equation}
where $\FF_A\subset \NN_A$ is a fundamental polyogon for $\Gamma_A\backslash \NN_A$ considered as a partial covering of $Y_0(1)$ via (\ref{eq:BBsubset}). We refer to \cite[Sec.\ 3]{DuImTo} for details. 
This turns out to be a special case of our construction using the following fundamental polygon for $\PSL_2(\Z)$:
\begin{equation}\label{eq:Fstd}\FF_0:=\{z\in \Hb: |z-1|> 1,0< \Re z< 1/2\},\end{equation}
 with vertices $\{0, \tfrac{1+i\sqrt{3}}{2},\infty\}$ and side pairing transformations $\mathrm{sp}(\PSL_2(\Z),\FF_0)=\{S,TS, ST^{-1}\}$. We encourage the reader to consult the informative illustrations \cite[Fig. 1-3]{DuImTo} for some  useful pictures to have in mind.  
  \begin{lemma}\label{lem:DITequiv} Let $K/\Q$ be a real quadratic extension and let $A\in \Cl_K^+$. Then it holds that    
$$  \FF_A=\PP(\CC_A,\FF_0)+nY_0(1)\in \mathrm{Partial}(Y_0(1)),$$
for some $n\in \Z_{\geq 0}$.  Here we consider $\FF_A$ as a partial covering of $Y_0(1)$ via the convention (\ref{eq:BBsubset}).   
  \end{lemma}
\begin{proof}
It follows from \cite[Sec. 3]{DuImTo} that the thin group $\Gamma_A$ has a fundamental domain of the shape: 
\begin{align}\label{eq:DITfd} \{z\in \Hb: 0< \Re z < m\}\setminus \left(\cup_{j=0}^\ell B_j\right),\quad B_j=\{z\in \Hb: |z-m_j|\leq 1\},\end{align}
where $m_0=0$, $m_j=\sum_{k=1}^j n_k$ for $j>0$ and $m=m_\ell$ with $\omega=[\overline{n_1,\ldots , n_\ell}],n_j\geq 2$ the \emph{minus continuous fraction} of the attracting endpoint $\omega$ of the two-sided infinite geodesic containing (a \emph{reduced} lift of) $\CC_A$. The intersection between the domain (\ref{eq:DITfd}) and the convex core $\NN_A$ yields a fundamental domain $\FF_A$ for $\Gamma_A\backslash \NN_A$ whose (oriented) boundary $\tilde{\CC}_A$ projects to $\CC_A\subset Y_0(1)$. Now we write the boundary as a concatenation: 
$$\tilde{\CC}_A= \LL_1\ast \cdots \ast \LL_n=(\partial\NN_A)\cap \overline{\FF_A},$$
where each of the oriented geodesic segment $\LL_i$ is contained in a single $\Gamma$-translate of $\FF_0$ with its endpoints on the boundary. Note that if we remove the regions 
\begin{equation}\label{eq:regionsremoved}\left\{z+m_j\in \Hb: -1/2< \Re z< 1/2,|z|>1\right\},\end{equation} 
above each $B_j$ from (\ref{eq:DITfd}), then the remaining region $\PP_\mathrm{remain}$ is tiled by $\FF_0$. If $\LL_i$ is contained in $\PP_\mathrm{remain}$ then the $\Gamma$-translate $\gamma_i\FF_0$ containing $\LL_i$ is clearly contained in the fundamental domain (\ref{eq:DITfd}) and contains no other $\LL_{i'}$. From this we conclude that $\gamma_i\FF_0\cap \FF_A$ is exactly a polygon of the shape (\ref{eq:PPi}) (defined relative to $\CC_A$ and $\FF_0$), noticing that the orientations indeed match up. Now if $\LL_i$ is contained in one of the regions (\ref{eq:regionsremoved}) then by construction $\LL_i$ intersects $\left\{z+m_j\in \Hb: |z|=1\right\}$ and there is another $\LL_{i'}$ also contained in the same region (here we identify $B_0$ and $B_\ell$ via $T^m\in \Gamma_A$), see \cite[eq.\ (3.6)]{DuImTo}. It follows that by gluing the two halves: $$\FF_A\cap \{z+m_j\in \Hb: 0< \pm \Re z < 1/2, |z|>1\},$$ by applying $T^{-m_j}ST^{m_j}$ for $j\neq 0,\ell$ and $ST^m, T^{-m}S$ for $j=0,\ell$, we obtain another polygon defined in terms of ``painting to the left'' relative to $\FF_0$, i.e.\ of the shape (\ref{eq:PPi}). This corresponds exactly to reversing the ``cutting up'' of the closed geodesic as in \cite[Figure 7]{DuImTo}. This yields the wanted conclusion.          
\end{proof}
Notice that in general one might not be able to glue together $\Gamma$-translates of the polygons $\PP_1,\ldots, \PP_m$ (making up the partial covering $\PP(\CC,\FF)$) to make a convex polygon inside $\Hb$ (see e.g. Figure \ref{fig:2}). From this point of view, the extra copies of the fundamental polygon $\FF_0$ in the partial covering defined by Duke--Imamo\={g}lu--T\'{o}th are there exactly to make this happen. 

In some cases the partial coverings defined in this paper do actually glue up to a convex polygon, as we will see in the next example.  
\subsection{Comparing with the construction of Humphries--Nordentoft}\label{sec:humnor}
In this section we will sketch how the partial coverings obtained in \cite{HumphriesNordentoft22} are also a special case of the construction of this paper. We will briefly recall the setting and refer to \emph{loc. cit.} for details. Let $q$ be prime, let $\Gamma_0(q)$ be the Hecke congruence subgroup of level $q$ and put $Y_0(q):=\Gamma_0(q)\backslash \Hb$. Let $\FF_q$ be a \emph{special fundamental polygon} of $\Gamma$ in the sense of Kulkarni \cite{Kulkarni91}, i.e.\ a fundamental polygon with a minimal number of sides. Let $K/\Q$ be a real quadratic field in which $q$ splits and let $A\in \Cl_K^+$ be an element of the narrow class group. Peter Humphries and the first named author defined in \cite{HumphriesNordentoft22} a thin subgroup $\Gamma_A(q)\subset \PSL_2(\Z)$ with convex core  $\NN_A(q)$ and an explicit fundamental polygon $\FF_A(q)\subset \Hb$ for $\Gamma_A(q)\backslash \NN_A(q)$. One obtains a partial covering by considering the image (counted with multiplicity) of $\FF_A(q)$ under $\pi_{\Gamma_0(q)}:\Hb\rightarrow Y_0(q)$. This construction is based on the geometric coding with respect to the side pairing transformations of $\FF_q$ of a hyperbolic matrix in the conjugacy class of $\Gamma_0(q)$ corresponding to $\CC_A(q)$. It satisfies that the boundary of $\FF_A(q)$ projects to the closed geodesic $\CC_A(q)\subset Y_0(q)$ associated to $A\in \Cl_K^+$ as defined in e.g. \cite{Darmon94}.
\begin{lemma}\label{lem:HuNoequiv}
 Let $q$ be a prime. Let $K/\Q$ be a real quadratic extension in which $q$ splits and let $A\in \Cl_K^+$. Then it holds that    
$$ \FF_A(q)=\PP(\CC_A(q),\FF_q)+nY_0(q)\in \mathrm{Partial}(Y_0(q)), $$
where $n\in \Z_{\geq 0}$ satisfies $n=0$ unless $\CC_A(q)$ passes through a side of $\FF_q$ labeled by an elliptic element of order three. Here we consider $\FF_A(q)$ as a partial covering of $Y_0(q)$ via the convention (\ref{eq:BBsubset}). 
\end{lemma}
\begin{proof}
By the construction of the fundamental domain of $\FF_A(q)$ in \cite[Sec. 3.2.1]{HumphriesNordentoft22} the boundary component of $\partial \FF_A(q)$ corresponding to $\CC_A(q)$ consists of a single geodesic segment $\tilde{\CC}\subset \Hb$. Let $\tilde{\mathcal{S}}$ denote the oriented (two-sided) infinite geodesic containing $\tilde{\CC}$ (i.e.\ the axis). For each intersection of $\tilde{\CC}$ with a $\Gamma_0(q)$-translate $\gamma_i\FF_q$ we see that by construction $\FF_A(q)$ contains the intersection of $\gamma_i\FF_q$ with the interior of $\tilde{\mathcal{S}}$. These are exactly the polygons as defined in (\ref{eq:PPi}) using the lift  $\tilde{\CC}$. Furthermore, $\FF_A(q)$ equals the union of all of these polygons in the case where $\CC_A(q)$ does not pass through a side of  $\FF_q$ with an elliptic label of order $3$. This yields the wanted equality. Note that when $\CC_A(q)$ does pass through an order $3$ side and then afterwards exits the $\Gamma$-translate of $\FF_q$ ``to the right'', then  $\PP(\CC_A(q),\FF_q)$ will be strictly smaller than $\FF_A(q)$. \end{proof}
\subsection{Equidistribution of partial coverings}
Recall the definition of \emph{covering evenly} from eq.\ (\ref{eq:coverevenly}) of the introduction:
\begin{defi} We say that a sequence $\PP_1,\PP_2,\ldots $  of partial coverings of $Y_\Gamma$ with positive volumes \emph{cover $Y_\Gamma$ evenly} if for any continuous and bounded function $\phi:Y_\Gamma\rightarrow \C$ it holds that 
\begin{equation}\label{eq:coverevenlydef}\frac{\int_{\PP_n}\phi}{\vol(\PP_n)}\rightarrow  \frac{\int_{Y_\Gamma} \phi}{\vol(\Gamma)},\quad \text{as }n\rightarrow \infty. \end{equation}
\end{defi}

\begin{lemma}\label{lem:Zelditch}
It suffices to check the convergence (\ref{eq:coverevenlydef}) for every smooth and compactly supported function $\phi:Y_\Gamma\rightarrow \C$.
\end{lemma}
\begin{proof} This follows from the arguments in the proof of \cite[Cor.\ 6.5]{Zelditch92}: Assume that (\ref{eq:coverevenlydef}) holds for all smooth and compactly supported functions. Then by density in the $L^\infty$-norm we conclude that (\ref{eq:coverevenlydef}) holds for all continuous and compactly supported functions. Now let $\phi$ be continuous and bounded. Then by multiplying by smooth bump functions of exhausting support and estimating the differences using boundedness, equidistribution and that we are dealing with probability measures, one gets the wanted. We will skip the details.\end{proof}
In order to compare to the results in \cite{DuImTo} and \cite{HumphriesNordentoft22}, we recall the formulation of equidistribution in terms of \emph{continuity sets}, i.e.\ measurable sets $\BB\subset Y_\Gamma$ such that $\mu(\partial \BB)=0$. Recall the notation from eq. (\ref{eq:volumeintersection}).
\begin{lemma}\label{lem:B}
A sequence of partial coverings $\PP_1,\PP_2,\ldots$ with positive volumes cover $Y_\Gamma$ evenly if and only if for every fixed continuity set $\mathcal{B}\subset Y_\Gamma$ it holds that 
$$\frac{\vol(\PP_n \cap \mathcal{B})}{\vol(\PP_n)}\rightarrow \frac{\vol(\mathcal{B})}{\vol(\Gamma)},\quad \text{as }n\rightarrow \infty.$$
\end{lemma}
\begin{proof}
Let $\mathcal{B}\subset Y_\Gamma$ be a continuity set. Then by a standard approximation argument we can find for every $\eps>0$ smooth functions $\phi_1,\phi_2:Y_\Gamma\rightarrow [0,1]$ such that $\phi_1 \leq \mathbf{1}_\mathcal{B} \leq \phi_2$ and $0\leq \mu(\phi_2-\phi_1)\leq \eps$. This shows one implication by monotonicity. Similarly, it is standard to show that the span of indicator functions of continuity sets are dense in $L^\infty(Y_\Gamma)$, which contains the space of smooth and compactly supported function on $Y_\Gamma$. Thus the other implication follows from Lemma \ref{lem:Zelditch} above. 
\end{proof}

\section{Volumes of partial coverings}\label{sec:lowerbound} As above let $\FF$ be a fundamental polygon a discrete and cofinite subgroup $\Gamma\leq \PSL_2(\R)$. A key step in our proof will be a lower bound for the volume of $\PP(\CC,\FF)$ as well as a non-escape of mass result in the non-compact case. In both \cite{DuImTo} and \cite{HumphriesNordentoft22} the volume estimates were achieved for individual geodesics $\CC$ using number theoretic input. Note that the non-escape of mass was not explicitly addressed in either paper but was implicit in the treatment of the Eisenstein spectrum (requiring a regularization procedure in the setting of \cite{HumphriesNordentoft22}). For general discrete and cofinite Fuchsian groups these arguments do not seem to be available. Instead we will obtain lower bounds \emph{on average} using equidistribution of the (oriented) closed geodesics in the unit tangent bundle combined with geometric considerations.
\subsection{Lower bound for the volume}
In this section we will obtain lower bounds for the volumes of partial coverings associated to equidistributing collections of closed geodesics. The basic idea is quite simple: equidistribution ensures that the geodesics cannot spend too much time close to the boundary of a(ny) fundamental polygon. In the following we make this idea quantitative. 

As above let $G=\PSL_2(\R)$ and consider the canonical projection 
\begin{equation}p:G\rightarrow G/\mathrm{PSO}_2\cong \Hb.\end{equation}
Given two non-empty, connected, open subsegments $I,J\subset \partial\FF$ of the boundary of $\FF$, we define the following subset of $G$:
$$\mathcal{G}(I,J):=\{g\in G: \exists \ell_1,\ell_2>0\text{ s.t. }p(ga_{-\ell_1})\in I, p(ga_{\ell_2})\in J\}\subset p^{-1}(\FF),$$
where the last inclusion follows by convexity of $\FF$.
\begin{lemma} If $I,J\subset \partial\FF$ are not both contained in a single edge of $\FF$ then 
$$\mu_{G}(\mathcal{G}(I,J))>0.$$
\end{lemma}
\begin{proof}
First of all observe that $\mathcal{G}(I,J)$ is open since $I$ and $J$ are open. The assumption that $I,J$ are not contained in a single edge together with the fact that $\FF$ is convex implies that there is some geodesic segment contained in $\FF$ connecting $I$ and $J$. Since $\mathcal{G}(I,J)$ is a non-empty open subset of $G$, the claimed lower bound follows from the fact that $\mu_{G}$ has full support.
\end{proof}
Given $g\in \mathcal{G}(I,J)$ we denote by $\ell_I(g),\ell_J(g)> 0$ the unique real numbers defined by 
$$p(g a_{-\ell_I(g)})\in I,\quad  p(ga_{\ell_J(g)})\in J,$$
and define
$$c_{I,J}:=\frac{\mu_{G}(\mathcal{G}(I,J))}{\max_{g\in \mathcal{G}(I,J)}(\ell_I(g)+\ell_J(g))}\geq 0.$$
Note that $\ell_I$ and $\ell_J$ are both continuous as a function of $g$.
\begin{lemma}
If $I,J\subset \partial\FF$ are not both contained in a single edge of $\FF$ \emph{and} neither the closure of $I$ nor of $J$ contains a cusp of $\Gamma$, then $c_{I,J}>0$.
\end{lemma}
\begin{proof}
By the above lemma we know that $\mu_{G}(\mathcal{G}(I,J))>0$. By the assumptions on $I,J$ it follows that the convex hull of $\overline{I\cup J}\subset \Hb$ is a compact subset of the closure of $\FF$. Thus we conclude by continuity of $\ell_I$ and $\ell_J$ that
$$\max_{g\in \mathcal{G}(I,J)}(\ell_I(g)+\ell_J(g))<\infty,$$  
which yields the wanted.
\end{proof}
Given $I,J\subset \partial\FF$ as above, let $z_0\in \overline{I},z_1\in \overline{J}$ denote the endpoints of $I$ and $J$ such that the segment of $\partial\FF$ connecting $z_0$ with $z_1$ in clockwise direction is disjoint from $I\cup J$.  Let $\mathcal{S}\subset \Hb$ be the unique oriented (two-sided) infinite geodesic containing $z_0$ and $z_1$ with endpoints $x_0,x_1\in \mathbf{P}^1(\R)$ and let 
$$\mathcal{B}=\gamma^{-1}\{z\in\Hb: \Re z\geq 0\},\quad\text{ with } \gamma \mathcal{S}=i\R, \gamma x_0=\infty, \gamma x_1=0,$$
be the interior of $\mathcal{S}$ (as defined in eq. (\ref{eq:tildeB})).  Then we get a hyperbolic polygon: 
$$\PP_{I,J}:=\mathcal{B}\cap \FF\subset \Hb,$$
and we put 
$$d_{I,J}:=\vol (\PP_{I,J})\geq 0,$$
suppressing the dependence on $\FF$ in the notation.
\begin{lemma}
Let $I,J$ be disjoint, non-empty, open, connected subsegments of $\partial \FF$ such that neither of the endpoints of $I$ (i.e.\ the boundary of $I$) lies on the same edge as any of the endpoints of $J$. Then $d_{I,J}>0$.
\end{lemma}
\begin{proof}
By the assumptions on $I,J$ we see that the endpoints $z_0,z_1$ defined above lie on two different edges of $\FF$. Thus $\PP_{I,J}$ is non-degenerate and the conclusion follows. 
\end{proof}
Finally, we define the following key invariant of $\FF$:
$$c(\Gamma,\FF):= \sup_{I,J\subset \partial \FF}c_{I,J}\cdot d_{I,J},$$
where the $\sup$ is taken over pairs $I,J$ of disjoint, non-empty, open, connected subsegments of $\partial \FF$. 
\begin{cor}
For any discrete and cofinite $\Gamma\leq \PSL_2(\R)$ and any fundamental polygon $\FF$ of $\Gamma$ we have 
$$ c(\Gamma,\FF)>0. $$
\end{cor}
\begin{proof}
Clearly, $\FF$ has at least three distinct edges. Let $I,J$ be non-empty, open, connected subsets of the boundary $\partial \FF$ whose closures are contained in the  interior of two distinct edges of $\FF$. Then clearly the closures of $I$ and $J$ do not contain a cuspidal vertex of $\FF$ and $I$ and $J$ do not have boundary points  on the same edge. Thus we conclude by the above lemmas that indeed $c_{I,J}>0$ and $d_{I,J}>0$, which yields the wanted lower bound.
\end{proof}
We can now prove the following fundamental bound.
\begin{prop}\label{prop:vol}
Assume that $\Cc_1,\Cc_2,\ldots$ equidistribute in the unit tangent bundle $\Gamma\backslash G$, in the sense of Definition \ref{def:equid}. Then for every $\eps>0$ it holds that 
$$\vol(\PP(\Cc_n,\FF))\geq (c(\Gamma,\FF)-\eps+o_{\FF,\eps}(1))  |\Cc_n|,\quad \text{as }n\rightarrow \infty.$$
\end{prop}
\begin{proof}
Let $I,J$ be disjoint, non-empty, connected, open subsegments of the boundary of $\FF$. We will consider $\mathcal{G}(I,J)$ as a subset of $\Gamma\backslash G$ via the inclusion $\mathcal{G}(I,J)\subset p^{-1}(\FF)$. Let 
$$\ell_n:=|\Cc_n|, $$
denote the total geodesic length of the packet $\Cc_n$ and consider a parametrization 
$$x_n:[0,\ell_n]\rightarrow \Gamma\backslash G,$$
of the concatenation of the entire collection $\Cc_n$ with unit speed (which is piecewise geodesic with at most $\#\Cc_n$ many discontinuity points). By the assumption that $(\Cc_n)_{n\geq 1}$ equidistributes in unit tangent bundle $\Gamma\backslash G$, we conclude that 
$$\mu_\mathrm{Borel}(\{ t\in[0,\ell_n]: x_n(t)\in \mathcal{G}(I,J) \})=
(\mu_{G}(\mathcal{G}(I,J))+o_{I,J}(1)) \ell_n,\quad \text{as }n\rightarrow \infty, $$
where $\mu_\mathrm{Borel}$ denotes the standard Borel measure on $\R$. This implies that the total number of times the collection of geodesics $\Cc_n$ (when projected to $\FF$) will connect $I$ and $J$ is at least 
$$\frac{\mu_\mathrm{Borel}(\{ t\in[0,\ell_n]: x_n(t)\in \mathcal{G}(I,J) \})}{\max_{g\in \mathcal{G}(I,J)}(\ell_I(g)+\ell_J(g))}=(c_{I,J}+o_{I,J}(1)) \ell_n,\quad \text{as }n\rightarrow \infty .$$ 
For every geodesic segment of $\Cc_n$ connecting $I$ with $J$, the associated polygon $\PP_i$, as in the definition (\ref{eq:PPdef}) of the partial cover $\PP(\CC,\FF)$, contains $\PP_{I,J}$ by convexity. Thus we conclude for any such $I,J$ that 
$$\vol(\PP(\Cc_n,\FF))\geq (c_{I,J}\cdot d_{I,J}+o_{I,J,\FF}(1))\ell_n,\quad \text{as }n\rightarrow \infty.$$
This implies the wanted lower bound for any $\eps>0$ by the definition of $c(\Gamma,\FF)$.
\end{proof}

\subsection{Non-escape of mass}
In the non-compact case special care has to be given to the parts of the partial coverings high in the cusps. This is a well-known phenomena in dynamical systems (see e.g. \cite[Sec.\ 5]{EinLindMichVenk12}) and is often referred to as \emph{non-escape of mass}. In our case this follows by a geometric argument comparing hyperbolic lengths and volumes. To set it up, let $Y>0$ and define the \emph{horostrip of height $Y$} as
\begin{equation}
\mathcal{H}(Y):= \{z\in \Hb: 0< \Re z< 1, \Im z\geq Y \},
\end{equation}
whose volume is given by
\begin{equation}\label{eq:horostrip}
\vol(\HH(Y))= \int_Y^{\infty}\int_0^1 y^{-2}dxdy=Y^{-1} .
\end{equation}
\begin{lemma}\label{lem:geometric}
Let $\mathcal{S}\subset \Hb$ be an infinite two-sided oriented geodesic and let $\mathcal{B}\subset \Hb$ denote the interior of $\mathcal{S}$ as in (\ref{eq:tildeB}).  Then for $0<\sigma<1$ and $Y\geq 2^{1/(1-\sigma)}$ it holds that 
\begin{equation}\label{eq:claim}\vol(\mathcal{B}\cap \mathcal{H}(Y))\leq 2 |\mathcal{S}\cap\mathcal{H}(Y^{\sigma})|+2 Y^{-(1-\sigma)} \vol(\mathcal{B}\cap \mathcal{H}(Y^{\sigma})).   \end{equation} 
\end{lemma}
\begin{proof}
Note that by a simple geometric consideration the intersection 
\begin{equation}\label{eq:intersection}\mathcal{S}\cap \partial \HH(Y^{\sigma}),\end{equation}
has cardinality less than or equal to two (it is the intersection of a half-circle and a rectangle with a side missing parallel to the diameter). In the case where the intersection (\ref{eq:intersection}) consists of two points we denote these by $z_1=x_1+iy_1$ and $z_2=x_2+iy_2$ where $0\leq x_1<x_2\leq 1$ and when the intersection consists of at most one point we put $x_1=x_2=1$ and $y_1=y_2=Y^{\sigma}$. Now consider the following partition of $\BB\cap \HH(Y^{\sigma})$:
\begin{align}
\nonumber\PP_{1}&:=\BB\cap \HH(Y^{\sigma})\cap \{z\in \Hb: 0\leq \Re z\leq x_1 \},\\
\nonumber \PP_{2}&:=\BB\cap \HH(Y^{\sigma})\cap \{z\in \Hb: x_1\leq \Re z\leq x_2 \},\\
 \PP_{3}&:=\BB\cap \HH(Y^{\sigma})\cap \{z\in \Hb: x_2\leq \Re z\leq 1 \},
\end{align}   
This is illustrated in Figures \ref{fig:escape} and \ref{fig:escape2}.
\begin{figure}
\begin{tikzpicture}[scale=2.5]
\clip (-1.3,-0.2) rectangle (1.3,2);
\fill[white] ((-1.3,-0.2) rectangle (1.3,2);
\begin{scope}
  \clip (-1,0.7) rectangle (1,2);
\fill[black,opacity=0.2] (-1,0.7) rectangle (1,2);
\end{scope}
\begin{scope}
  \clip (-1,0.7) rectangle (-0.863,2);
\fill[black,opacity=0.2, pattern=north east lines
] (-1,0.7) rectangle (-0.863,2);
\end{scope} 
\begin{scope}
  \clip (-0.863,0.7) rectangle (0.463,2);
\fill[black,opacity=0.2, pattern=crosshatch] (-0.863,0.7) rectangle (0.463,2);
\end{scope} 
\begin{scope}
  \clip (0.463,0.7) rectangle (1,2);
\fill[black,opacity=0.2, pattern=north west lines] (0.463,0.7) rectangle (1,2);
\end{scope}

\begin{scope}
  \clip (-1,0) rectangle (1,2);
\fill[white] (-0.2,-0.3) circle(1.2);
\draw[decoration={markings,
      mark=between positions 0 and 1 step 14pt with {\arrowreversed{to}}},
    postaction={decorate}] (-0.2,-0.3) circle (1.2);
\end{scope}
\draw[very thin] (-0.863,0)--(-0.863,2);
\draw[very thin] (0.463,0)--(0.463,2);
 \draw[thin] (-1,0)--(-1,2);
  \draw[thin] (1,0)--(1,2);
  \draw[dotted] (1.1,0.7)--(-1.05,0.7);
  \node at (-1.15,0.7) {\scalebox{0.55}{$Y^{\sigma}$}};
  \node at (1.2,0.7) {\scalebox{0.55}{$y_1$}};
 \draw[dotted] (1.05,1.7)--(-1.15,1.7);
  \node at (-1.2,1.7) {\scalebox{0.55}{$Y$}};
  \node at (0.463,0) [below] {{\tiny\color{gray} $ x_2$}};
   \node at (-0.863,0) [below] {{\tiny\color{gray} $ x_1$}}; 
    \node at (-1,0) [below] {{\tiny\color{gray} $ 0$}}; 
      \node at (1,0) [below] {{\tiny\color{gray} $ 1$}};
        \node at (-0.93,1) {\scalebox{.5}{$\mathcal{P}_{1}$}};
        \node at (-0.2,1.3) {$\mathcal{P}_{2}$};
        \node at (0.75,1.2) {$\mathcal{P}_{3}$};
         \node at (0.73,0.3) {\tiny $\mathcal{S}$};
\end{tikzpicture}\qquad
\begin{tikzpicture}[scale=2.5]
\clip (-1.3,-0.2) rectangle (1.2,2);
\fill[white] ((-1.3,-0.2) rectangle (1.2,2);
\begin{scope}
  \clip (-1,0.7) rectangle (1,2);
\fill[black,opacity=0.2] (-1,0.7) rectangle (1,2);
\end{scope}

\begin{scope}
  \clip (-1,0) rectangle (1,2);
\fill[white] (0.3,-0.3) circle (1.9);
\draw[decoration={markings,
      mark=between positions 0 and 1 step 14pt with {\arrowreversed{to}}},
    postaction={decorate}] (0.3,-0.3) circle (1.9);
\end{scope}
 \draw[thin] (-1,0)--(-1,2);
  \draw[thin] (1,0)--(1,2);
  \draw[dotted] (1.05,0.7)--(-1.05,0.7);
  \node at (-1.15,0.7) {\scalebox{0.55}{$Y^{\sigma}$}};
 \draw[dotted] (1.05,1.7)--(-1.15,1.7);
  \node at (-1.2,1.7) {\scalebox{0.55}{$Y$}};
   \draw[dotted] (1.05,1.0856)--(-1.15,1.0856);
  \node at (-1.2,1.0856) {\scalebox{0.55}{$y_1$}};
    \node at (-1,0) [below] {{\tiny\color{gray} $ x_1=0$}}; 
      \node at (1,0) [below] {{\tiny\color{gray} $ x_2=1$}};
        \node at (0,1.7) {$\mathcal{P}_{2}$};
         \node at (0.5,1.45) {\tiny $\mathcal{S}$};
\end{tikzpicture}
\caption{Examples of the partition of $\BB\cap \HH(Y^{\sigma})$ into $\mathcal{P}_{1},\mathcal{P}_{2},\mathcal{P}_{3}$ in the case where $\BB$ contains the cusp at $\infty$.}\label{fig:escape}
\end{figure}
\begin{figure}
\begin{tikzpicture}[scale=2.5]
\clip (-1.3,-0.2) rectangle (1.2,2);
\fill[white] ((-1.3,-0.2) rectangle (1.2,2);
\begin{scope}
  \clip (-0.762,0) rectangle (1,2);
\fill[black,opacity=0.2] (1.2,-0.5) circle(2.3);
\end{scope}
\draw[very thin] (-0.762,0)--(-0.762,2);
\begin{scope}
 \clip (-0.762,0) rectangle (1,0.7);
\fill[white] (-0.762,0) rectangle (1,0.7);
\end{scope}
\begin{scope}
  \clip (-1,0) rectangle (1,2);
\draw[decoration={markings,
      mark=between positions 0 and 1 step 14pt with {\arrow{to}}},
    postaction={decorate}] (1.2,-0.5) circle(2.3);
\end{scope}
 \draw[thin] (-1,0)--(-1,2);
  \draw[thin] (1,0)--(1,2);
 \draw[dotted] (1.05,0.7)--(-1.05,0.7);
  \node at (-1.15,0.7) {\scalebox{0.55}{$Y^{\sigma}$}};
    \node at (1.15,0.7) {\scalebox{0.55}{$y_1$}};
  \draw[dotted] (1.05,1.7)--(-1.15,1.7);
  \node at (-1.2,1.7) {\scalebox{0.55}{$Y$}};
  \node at (-0.762,0) [below] {{\tiny\color{gray} $ x_1$}}; 
    \node at (-1,0) [below] {{\tiny\color{gray} $ 0$}}; 
      \node at (1,0) [below] {{\tiny\color{gray} $ x_2=1$}};
        \node at (0.3,0.9) {$\mathcal{P}_{2}$};
        \node at (-0.4,0.97) {\tiny $\mathcal{S}$};
\end{tikzpicture}\qquad
\begin{tikzpicture}[scale=2.5]
\clip (-1.3,-0.2) rectangle (1.2,2);
\fill[white] ((-1.3,-0.2) rectangle (1.2,2);
\begin{scope}
  \clip (-1,0.7) rectangle (1,2);
\fill[black,opacity=0.2] (0.1,-0.93) circle (2.9);
\end{scope}
\begin{scope}
  \clip (-1,0) rectangle (1,2);
\draw[decoration={markings,
      mark=between positions 0 and 1 step 14pt with {\arrow{to}}},
    postaction={decorate}] (0.1,-0.93) circle (2.9);
\end{scope}
 \draw[thin] (-1,0)--(-1,2);
  \draw[thin] (1,0)--(1,2);
  \draw[dotted] (1.05,0.7)--(-1.05,0.7);
  \node at (-1.15,0.7) {\scalebox{0.55}{$Y^{\sigma}$}};
 \draw[dotted] (1.05,1.7)--(-1.15,1.7);
  \node at (-1.2,1.7) {\scalebox{0.55}{$Y$}};
   \draw[dotted] (1.05,1.753)--(-1.05,1.753);
  \node at (1.1,1.753) {\scalebox{0.55}{$y_1$}};
    \node at (-1,0) [below] {{\tiny\color{gray} $ x_1=0$}}; 
      \node at (1,0) [below] {{\tiny\color{gray} $ x_2=1$}};
        \node at (0,1.5) {$\mathcal{P}_{2}$};
         \node at (0.6,1.85) {\tiny $\mathcal{S}$};
\end{tikzpicture}

\caption{Examples of $\PP_{2}=\BB\cap \HH(Y^{\sigma})$ in the case where $\BB$ does not contain the cusp at $\infty$ (so that $\PP_{1}=\PP_{3}=\emptyset$).}\label{fig:escape2}
\end{figure}
If $\mathcal{S}$ is oriented so that it connects $z_1$ with $z_2$ (so that $\BB$ contains the cusp at $\infty$) then by the volume formula for horostrips (\ref{eq:horostrip}) we conclude 
\begin{align*}
\vol(\PP_{1} \cap \HH(Y))=x_1 Y^{-1}=Y^{-(1-\sigma)}(x_1 Y^{-\sigma})&=Y^{-(1-\sigma)}\vol(\PP_{1})\\
&\leq Y^{-(1-\sigma)} \vol(\mathcal{B}\cap \mathcal{H}(Y^{\sigma})), 
\end{align*} 
and similarly for $\PP_{3} $ (see Figure \ref{fig:escape} where in the second example $\PP_{1}=\PP_{3}=\emptyset$). When bounding the volume of  $\PP_{2} $ we may reduce to the case where (\ref{eq:intersection}) has cardinality two and $Y^{\sigma}\leq y_1\leq y_2$. In this case we have
\begin{align}\label{eq:Pj1}
|\mathcal{S}\cap \HH(Y^{\sigma})|\geq \mathrm{dist}(z_1,x_2+i\R)=\log \left(\sqrt{\kappa^2+1}+\kappa\right)\geq \frac{1}{2} \kappa,\quad \kappa=\frac{x_2-x_1}{y_1}, 
\end{align} 
for $Y\geq 1$ (note that $0< x_2-x_1\leq 1$ and $y_1\geq Y^{\sigma}\geq 1$), using here a standard formula for the hyperbolic distance between a point and a line (see e.g. \cite[Sec. 7.20]{Beardon83}). On the other hand, we have trivially (as illustrated in Figure \ref{fig:escape}) that
\begin{equation}\label{eq:Pj2}\vol(\PP_{2}\cap \HH(Y))\leq\vol( \HH(y_1)\cap \{x_1\leq \Re z\leq x_2 \})= \frac{x_2-x_1}{y_1}, \end{equation}
which proves the bound (\ref{eq:claim}) in this case. If $\mathcal{S}$ is oriented so that it connects $z_2$ with $z_1$ (so that $\BB$ does not contain the cusp $\infty$), then we have $\PP_{1} =\PP_{3} =\emptyset$. We may reduce to the case where (\ref{eq:intersection}) has cardinality $2$ and $Y^{\sigma}\leq y_1\leq y_2$. If $y_1\leq Y$ (as in the first example in Figure \ref{fig:escape2}) then we have \begin{equation}\nonumber\vol(\PP_{2}\cap \HH(Y))\leq\vol( \HH(y_1)\cap \{x_1\leq \Re z\leq x_2 \})= \frac{x_2-x_1}{y_1}\leq 2 |\mathcal{S}\cap \HH(Y^{\sigma})|, \end{equation} 
for $Y\geq 1$, using the bound (\ref{eq:Pj1}) in the last inequality. If $y_1\geq Y$ then $x_0=0$ and $x_1=1$ (as in the second example in Figure \ref{fig:escape2}) and we have
\begin{equation}\nonumber
\vol(\BB\cap \HH(Y^\sigma)) \geq \vol(\HH(Y^{\sigma}))-\vol(\HH(Y))= Y^{-\sigma}-Y^{-1}\geq \frac{1}{2}Y^{-\sigma}, 
\end{equation}
whenever $Y\geq 2^{1/(1-\sigma)}$. This implies that
 \begin{equation}\nonumber
\vol(\PP_{2} \cap \HH(Y))\leq Y^{-1}\leq 2 Y^{-(1-\sigma)}\vol(\BB\cap \HH(Y^\sigma)), 
\end{equation}
which proves the required bound in this case as well.\end{proof}
With this geometric lemma at our disposal we are ready to prove non-escape of mass for equidistributing closed geodesics. Recall the notation (\ref{eq:volumeintersection}) and the definition of the cuspidal region from (\ref{eq:cuspidalzone}).
\begin{prop}\label{prop:nonescape}
Let $\Cc_1,\Cc_2,\ldots $ be a sequence of collections of oriented closed geodesics equidistributing in the unit tangent bundle $\Gamma\backslash \PSL_2(\R)$ and let $\mathfrak{a}\in \mathrm{cusp}(\FF)$ be a cusp of the fundamental polygon $\FF$. Then there exists constants $c>0$ and $Y_0>0$ such that for every $Y\geq Y_0$ there exists $n_0(Y)\geq 1$ such that the following holds: 
\begin{equation} 
\vol(\PP(\Cc_n,\FF)\cap \FF_\mathfrak{a}(Y))\leq c \, \vol(\PP(\Cc_n,\FF))Y^{-1/2}, \quad n\geq n_0(Y).
\end{equation}
Here all choices of constants are allowed to depend on the fundamental polygon $\FF$. 
\end{prop}
\begin{proof}
By conjugation we may reduce to the case 
$\mathfrak{a}=\infty$ with width one and by letting $Y$ be sufficiently large (depending on $\FF$) we may assume that 
\begin{equation}\label{eq:FFHH}\FF_\infty(Y)=\HH(Y)=\{z\in \Hb: 0< \Re z< 1, \Im z\geq  Y \},\end{equation}
with the cuspidal region defined in eq.\ (\ref{eq:cuspidalzone}), using here that all the choices are allowed to depend on $\FF$. Let $Y_0\geq 4$ be such that $\FF_\infty(Y^{1/2})=\HH(Y^{1/2})$ holds for $Y\geq Y_0$. For each $n\geq 1$ let 
$$\mathcal{S}_{1,n},\ldots, \mathcal{S}_{m_n,n},\quad  \text{and}\quad \PP_{1,n}=\BB_{1,n}\cap \gamma_{1,n}\FF,\ldots, \PP_{m_n,n}=\BB_{m_n,n}\cap \gamma_{m_n,n}\FF,$$ 
denote the infinite geodesics and polygons as in (\ref{eq:PPi}) used to define the partial covering $\PP(\Cc_n,\FF)$ (note that the $\mathcal{S}_{j,n}$'s and  $\BB_{j,n}$'s are not necessarily  all distinct). For $Y\geq Y_0$ we have 
\begin{align}\label{eq:Pij}(\gamma_{j,n}^{-1}\BB_{j,n})\cap \HH(Y^{1/2})= (\gamma_{j,n}^{-1}\PP_{j,n})\cap \FF_\infty(Y^{1/2}),\quad j=1,\ldots, m_n.\end{align} 
Now we partition $\PP(\Cc_n,\FF)\cap \FF_\infty(Y)$ in terms of the $\PP_{j,n}$'s and apply Lemma \ref{lem:geometric} with $\sigma=1/2$ to the infinite oriented  geodesic $\gamma_{j,n}^{-1}\mathcal{S}_{j,n}$ for each $j=1,\ldots, m_n$ which in view of (\ref{eq:Pij}) gives
\begin{align}\nonumber
\vol(\PP(\Cc_n,\FF)\cap \FF_\infty(Y))&=\sum_{j=1}^{m_n} \vol\left((\gamma_{j,n}^{-1}\PP_{j,n})\cap \FF_\infty(Y)\right)\\
\nonumber &\leq 2\sum_{j=1}^{m_n} \biggr(\left|(\gamma_{j,n}^{-1}\mathcal{S}_{j,n})\cap\HH(Y^{1/2})\right|+Y^{-1/2} \vol\left((\gamma_{j,n}^{-1}\BB_{j,n})\cap\HH(Y^{1/2})\right)\biggr) \\
\nonumber &\leq 2\sum_{j=1}^{m_n} \biggr(|(\gamma_{j,n}^{-1}\mathcal{S}_{j,n})\cap\HH(Y^{1/2})|+Y^{-1/2} \vol(\PP_{j,n})\biggr) \\
\label{eq:boundnoescape}&= 2|\Cc_n\cap\FF_\infty(Y^{1/2})|+2Y^{-1/2} \vol(\PP(\Cc_n,\FF)),\quad Y\geq Y_0,
\end{align}
using also the trivial inequality $\vol((\gamma_{j,n}^{-1}\PP_{j,n})\cap\FF_\infty(Y^{1/2}))\leq \vol(\PP_{j,n})$. By equidistribution of $(\Cc_n)_{n\geq 1}$ we have for any $Y>0$ that 
$$\frac{|\Cc_n\cap\FF_\infty(Y^{1/2})|}{|\Cc_n|}=\frac{\vol(\FF_\infty(Y^{1/2}))}{\vol(\Gamma)}+o_{\FF, Y}(1),\quad \text{as }n\rightarrow \infty.$$
Combined with Proposition \ref{prop:vol} and the volume formula (\ref{eq:horostrip}) we conclude that for $n$ sufficiently large (depending on $Y$ and $\FF$) it holds that  
$$|\Cc_n\cap\FF_\infty(Y^{1/2})|\leq 2 \frac{\vol(\FF_\infty(Y^{1/2}))}{\vol (\Gamma)}|\Cc_n|\ll_\FF Y^{-1/2}\vol(\PP(\Cc_n,\FF)) .$$
This finishes the proof in view of (\ref{eq:boundnoescape}).
 \end{proof}

\section{Equidistribution}
Our main theorem is that equidistribution of partial coverings associated to oriented closed geodesics is a consequence of the equidistribution of the same geodesics in the unit tangent bundle.
\begin{thm}\label{thm:main}
Let $\FF$ be a fundamental polygon for a discrete and cofinite subgroup $\Gamma\leq \PSL_2(\R)$. Assume that we are given a sequence of collections of oriented closed geodesics $\Cc_1,\Cc_2,\ldots$ which equidistribute with respect to Haar measure in the unit tangent bundle $\Gamma\backslash G$, i.e.\ for continuous and bounded function $\Psi:\Gamma\backslash G\rightarrow \C$ it holds that  
\begin{equation}\label{eq:equidigeodesics}\frac{\sum_{\CC\in \Cc_n}\mu_\CC(\Psi)}{|\Cc_n|}\rightarrow  \frac{\int_{\Gamma\backslash G} \Psi(g)d\mu_G(g)}{\vol(\Gamma)},\quad \text{as $n\rightarrow \infty$.} \end{equation}
Then the partial coverings $\PP(\Cc_n,\FF)$ cover $Y_\Gamma$ evenly
i.e.\ for any continuous and bounded function $\phi:Y_\Gamma\rightarrow \C$ it holds that 
\begin{equation}\label{eq:equidisurface}\frac{\int_{\PP(\Cc_n,\FF)}\phi}{\vol(\PP(\Cc_n,\FF))}\rightarrow   \frac{\int_{Y_\Gamma} \phi}{\vol(\Gamma)},\quad \text{as $n\rightarrow \infty$.} \end{equation}
Moreover, if $\Gamma$ is co-compact and the convergence (\ref{eq:equidigeodesics}) holds with an effective error-term $O_\Psi(E(n))$ with $E(n)\rightarrow 0,n\rightarrow \infty$ then (\ref{eq:equidisurface}) holds with the same effective error-term (possibly with a different implied constant allowed to depend on $\FF$).   
\end{thm}

The proof will occupy the rest of this section. 

\subsection{An application of Stokes' theorem}
The first key step, going back to Duke--Imamo\={g}lu--T\'{o}th \cite[Lemma 2]{DuImTo}, is an application of Stokes' theorem which allows one to reduce the integral over the partial coverings to an integral over the boundary of a weight $2$ automorphic form. We present here a slightly more automorphic perspective in terms of raising and lowering operators (\ref{eq:levelraising}). We begin with an elementary lemma.\begin{lemma}\label{lem:deriv}
Let $\varphi: \C\rightarrow \C$ be a $1$-differentiable function. Then it holds that
\begin{equation}
L_2 y \varphi(z)=2i y^2 \partial_{\overline{z}} \varphi(z), 
\end{equation}
 where $L_2=1+2iy\partial_{\overline{z}}=1+iy\frac{\partial }{\partial x}-y\frac{\partial }{\partial y}$ denotes the lowering operator of weight $2$.
\end{lemma}
\begin{proof}
This follows by direct computation.
\end{proof}
Recall next the definition (\ref{eq:bnd}) and (\ref{eq:coho}) of the bases $\CC_{1}(\FF),\ldots, \CC_{2g}(\FF)$ and $\omega_{1}(\FF),\ldots, \omega_{2g}(\FF)$ for homology and cohomology.
\begin{lemma}[Cf. Lemma 2 of \cite{DuImTo}]\label{lem:Stokes}Let $\Phi\in \mathcal{A}(\Gamma,2)$ be an automorphic form for $\Gamma$ of weight $2$. If $\Gamma$ is not co-compact assume further that $\Phi$ and $L_2 \Phi$ are both rapidly decaying at all cusps of $\Gamma$. Then it holds that
\begin{equation} \int_{\PP(\CC,\FF)} L_2 \Phi=  \int_\CC \Phi(z) \tfrac{dz}{\Im z} - \sum_{i=1}^{2g} \left(\int_{\CC_{i}(\FF)} \Phi(z) \tfrac{dz}{\Im z}\right) \, \int_{\CC} \omega_{i}(\FF).  \end{equation}
\end{lemma}
\begin{proof}
By Lemma \ref{lem:deriv} we see that $L_2 \Phi=2iy^2 \partial_{\overline{z}} \varphi$ where $\Phi(z)=(\Im z) \varphi(z)$. Recall also that $dz\wedge\overline{dz}=-2i\, dxdy$. In the co-compact case an application of Stokes' theorem yields directly: 
\begin{align}
\int_{\PP(\CC,\FF)} L_2 \Phi(z) \frac{dxdy}{y^2}  = -\int_{\PP(\CC,\FF)}  \partial_{\overline{z}} \varphi(z) dz\wedge d\overline{z}
&= \int_{\partial\PP(\CC,\FF)}  \varphi(z) dz\\
&= \int_{\partial\PP(\CC,\FF)} \Phi(z) \tfrac{dz}{\Im z}, \end{align}
The same remains true in the non-co-compact case by a truncation argument as in proof of \cite[Lemma 2]{DuImTo}: by the rapid decay assumption on $\varphi$ we can consider $\varphi(z)dz$ as a bounded $1$-form on $\PP(\CC,\FF)$  whereas the length of a closed horocycle in $Y_\Gamma$ tends toward $0$ as it approaches a cusp. 

By construction we can partition the boundary $\partial\PP(\CC,\FF)$ in terms of the boundary of the polygons $\PP_1,\ldots, \PP_m$ given by the recipe (\ref{eq:PPi}) so that: 
$$\partial\PP(\CC,\FF)= \cup_{i=1}^m \partial \PP_i = \tilde{\CC}\cup \tilde{\LL}\subset \Hb, $$
where $\tilde{\LL}\subset \Gamma(\partial \FF)$ is the part contained in the $\Gamma$-translates of the boundary of $\FF$. Since the projection $\pi_\Gamma(\tilde{\CC})=\CC\subset Y_\Gamma$, with $\pi_\Gamma:\Hb\rightarrow Y_\Gamma$ as in (\ref{eq:piG}), is closed it follows that the projection 
$$\LL:=\pi_\Gamma(\tilde{\LL})\subset \Gamma\backslash \partial \FF, $$
is the image of a continuous closed curve and thus defines an element of $\Delta^\mathrm{cl}_1(\Gamma\backslash \partial \FF)$ (which we will denote by the same symbol). 
Since $\CC\cup \LL\subset X_\Gamma$ is the boundary of the $2$-chain $\pi_\Gamma(\cup_{i=1}^m \PP_i)$ we conclude that $\CC$ and $\LL$ are anti-homologous:
$$  [\LL]=-[\CC]\in H_1(X_\Gamma,\Z).$$ 
Recall that $\CC_1(\FF),\ldots, \CC_{2g}(\FF)\subset \Gamma\backslash \partial\FF$ is an integral basis for the $\Delta$-simplicial homology group $H_1^\Delta(\Gamma\backslash \partial\FF)$ of the boundary graph as defined in eq. (\ref{eq:bnd}). By definition (\ref{eq:coho}) of the dual basis $\omega_1(\FF),\ldots, \omega_{2g}(\FF)$ we conclude that 
\begin{equation}\label{eq:LL=-CC}[\LL]=-[\CC]=-\sum_{i=1}^{2g}[\CC_{i}(\FF)]\langle \CC, \omega_i(\FF)\rangle_P\in H_1(X_\Gamma,\Z).\end{equation}
By the isomorphism (\ref{eq:isohom}), this implies the following equality of closed $1$-simplices (considered as subsets of $Y_\Gamma$):
\begin{equation}\label{eq:LL=-sum}\LL= -\sum_{i=1}^{2g}\CC_{i}(\FF)\langle \CC, \omega_i(\FF)\rangle_P\in \Delta^\mathrm{cl}_1(\Gamma\backslash \partial \FF).\end{equation}
By eq. (\ref{eq:LL=-sum}) we arrive at 
\begin{align}\nonumber
\int_{\partial\PP(\CC,\FF)} \Phi(z) \tfrac{dz}{\Im z}&=\int_{\CC} \Phi(z) \tfrac{dz}{\Im z}+\int_{\LL} \Phi(z) \tfrac{dz}{\Im z}\\
&=\int_{\CC} \Phi(z) \tfrac{dz}{\Im z}- \sum_{i=1}^{2g} \left(\int_{\CC_{i}(\FF)} \Phi(z) \tfrac{dz}{\Im z}\right) \langle \CC, \omega_i(\FF)\rangle_P.
\end{align}
Now the conclusion follows since $\langle \CC, \omega_i(\FF)\rangle_P=\int_\CC \omega_i(\FF)$ by standard properties of the Poincar\'{e} pairing.
\end{proof}
\begin{remark}The sums over $i=1,\ldots, 2g$ are what we above referred to as the \emph{topological terms}.  Note that all the dependence on the geodesic $\CC$ in the topological terms is in the Poincar\'{e} pairings $\int_\CC \omega_i(\FF)$. For us both $\Gamma$ and $\FF$ are fixed and so it will suffice that the terms $\int_{\CC_i(\FF)} \Phi$ are finite numbers. In \cite{HumphriesNordentoft22} the first named author and Peter Humphries studied sparse equidistribution for Hecke congruence groups $\Gamma_0(q)$ with $q$ varying which leads to the problem of bounding $\int_{\CC_i(\FF_q)} \Phi$ in terms of $q$. This requires resolving a homological version of the sup norm problem which was handled in \cite{NordentoftConcentration} by the first named author. It would be interesting to see what can be said in the general setting considered in the present paper.\end{remark}

\subsection{Finding the antiderivatives}
The preceding section leads to the problem of solving for each test function $\Psi\in \mathcal{A}(\Gamma, 0)$ the equation $L_2 \Phi=\Psi$ for $\Phi\in \mathcal{A}(\Gamma, 2)$ with good control on the regularity. Note that this is only possible when $\langle \Psi,1\rangle=0$. In this case we achieve the wanted solution with two different arguments depending on whether $\Gamma$ is co-compact or not. 
\begin{lemma}[Compact case]\label{lem:antiderivcpt}
Let $\Gamma\leq \PSL_2(\R)$ be a discrete and co-compact subgroup. Let $\Psi\in \mathcal{A}(\Gamma,0)$ be an automorphic form of weight $0$ satisfying $\langle \Psi,1\rangle=0$. Then there exists an automorphic form $\Phi\in \mathcal{A}(\Gamma, 2)$ of weight $2$ such that
 \begin{equation} L_2 \Phi= \Psi,\end{equation} 
 where $L_2=1+2iy\partial_{\overline{z}}=1+iy\frac{\partial }{\partial x}-y\frac{\partial }{\partial y}$ denotes the lowering operator of weight $2$.
\end{lemma}
\begin{proof}
Recall the basic identity $L_2R_0= \Delta$ from eq. (\ref{eq:Delta}). The idea, going back to \cite{DuImTo}, is to pick $\Phi= R_0 \Delta^{-1} \Psi$, where $\Delta^{-1}$ denotes the resolvent operator $R(s; \Delta)$ at $s=0$ which is well-defined since $\langle \Psi,1\rangle=0$. Concretely, by combining the spectral expansion (\ref{eq:spectralcompact}), the estimate (\ref{eq:boundIP}) and the convergence of (\ref{eq:spectralbound}) we have
\begin{align}
\Delta^{-1} \Psi= \sum_{f\in B(\Gamma,0):\lambda_f>0} \frac{\langle \Psi, f\rangle}{\lambda_f} f, 
\end{align}
with the sum converging absolutely and uniformly. Thus $\Delta^{-1} \Psi$ defines a smooth function on $X_\Gamma$ and it follows that $\Phi= R_0 \Delta^{-1} \Psi\in \mathcal{A}(\Gamma, 2)$ as wanted. 
\end{proof}

\begin{lemma}[Non-compact case]\label{lem:ODE}
Let $\Gamma\leq \PSL_2(\R)$ be a discrete and cofinite subgroup with a cusp at infinity of width one. Let $P_{m}(\cdot|\psi)\in C_c^\infty(X_\Gamma)$ be an incomplete Poincar\'{e} series as defined in eq. (\ref{eq:Poincare}) with $\psi\in C_c^\infty(\R_{>0})$ and $m\geq 0$ satisfying $\langle P_{m}(\cdot|\psi),1 \rangle=0$ (which is automatic for $m\neq 0$). Then there exists $\varphi: \R_{>0}\rightarrow \C$ which is smooth and rapidly decaying at $0$ and $\infty$ such that
 \begin{equation} L_2 P_{m,2}(z|\varphi )= P_{m}(z|\psi),\end{equation} 
 where $L_2=1+2iy\partial_{\overline{z}}=1+iy\frac{\partial }{\partial x}-y\frac{\partial }{\partial y}$ denotes the lowering operator of weight $2$.
\end{lemma}
\begin{proof} From the intertwining property (\ref{eq:intertwine}) we conclude for $\varphi:\R_{>0}\rightarrow \C$ of Schwarz class: 
$$L_2 P_{m,2}(z|\varphi)= P_{m}(z|-y\tfrac{\partial}{\partial y}\varphi-(2\pi m y-1)\varphi) .$$
Thus we have reduced the task to solving a first order ODE:
$$\tfrac{\partial}{\partial y}\varphi+(2\pi m-y^{-1})\varphi= -y^{-1}\psi\Leftrightarrow \varphi(y)=-e^{-2\pi m y} y\left(\int_0^y e^{2\pi m t}\frac{\psi(t)}{t^2} dt+c\right), c\in \C.$$
We claim that putting $c=0$ yields the desired solution. For $m> 0$ this follows by the rapid decay of $e^{-2\pi m y}$ at $\infty$ and the compact support of $\psi$. For $m=0$ we recall that by unfolding:
$$0=\langle P_{0}(\cdot|\psi),1 \rangle= \int_0^\infty \psi(y)\frac{dy}{y^2}, $$
using the orthogonality assumption, which implies that $\varphi$ has compact support (using again that $\psi$ is compactly supported). This completes the proof.  
\end{proof}

\subsection{Proof of main result}
We are now ready to prove the main result. To simplify notation we put 
$$\mu_n:=\frac{1}{\vol(\PP(\Cc_n,\FF))}\PP(\Cc_n,\FF) d\mu,\quad n\geq 1.$$  
with $d\mu$ the hyperbolic measure as defined in (\ref{eq:hyperbolic}). We begin by considering the compact case where the proof is straightforward given the above.
\begin{proof}[Proof of Theorem \ref{thm:main}, compact case]
Assume that $\Gamma$ is co-compact and let $\Psi\in \mathcal{A}(\Gamma,0)$ be a smooth function on $X_\Gamma$. Then by Lemma \ref{lem:antiderivcpt} there exists $\Phi\in \mathcal{A}(\Gamma,2)$ such that
\begin{align}
\mu_n(\Psi)=\frac{\langle \Psi,1\rangle}{\vol(\Gamma)}+\mu_n\left(\Psi-\frac{\langle \Psi,1\rangle}{\vol(\Gamma)}\right)
=\frac{\langle \Psi,1\rangle}{\vol(\Gamma)}+\mu_n(L_2 \Phi).
\end{align}
By Lemma \ref{lem:Stokes} above, we see that
\begin{equation}\label{eq:afterstoke}\mu_n(L_2 \Phi)=\frac{1}{\vol(\PP(\Cc_n,\FF))}\left( \int_{\Cc_n} \Phi(z)\tfrac{dz}{\Im z}-  \sum_{i=1}^{2g} \left(\int_{\CC_{i}(\FF)}\Phi(z) \tfrac{dz}{\Im z}\right) \, \int_{\Cc_n} \omega_{i}(\FF)\right). \end{equation}
Note that the integrals over $\CC_{i}(\FF)$ are independent of $n$. Recall from eq. (\ref{eq:DF}) that we can write the cohomology class $\omega_{i}(\FF)$ as a sum of holomorphic and anti-holomorphic $1$-forms: $(\Im z)f_i(z)\tfrac{dz}{\Im z}+\overline{(\Im z)g_i(z)\tfrac{dz}{\Im z}}$ where $(\Im z)f_i$, resp. $\overline{(\Im z)g_i}$, is a weight $2$, resp. $-2$, automorphic form for $\Gamma$. So by Lemma \ref{lem:weight2} the integrals over $\Cc_n$ in eq. (\ref{eq:afterstoke}) are exactly the measure $\sum_{\CC\in \Cc_n}\mu_{\CC}$ applied to test functions on the unit tangent bundle of weight $2$. Also recall from eq. (\ref{eq:kneq0}) that test functions of non-zero weight are orthogonal to the constant function. Now the assumption that $(\Cc_n)_{n\geq 1}$ equidistribute in the unit tangent bundle $\Gamma\backslash G$ applied to the test functions $\Phi, (\Im z)f_i, (\Im z)\overline{g_i}$ combined with the lower bound on the volume from Proposition \ref{prop:vol} yields:
$$ |\mu_n(L_2 \Phi)|\ll_\Gamma \frac{1}{|\Cc_n|}\left( \left|\int_{\Cc_n} \Phi(z)\tfrac{dz}{\Im z}\right|+\sum_{i=1}^{2g} \left|\int_{\CC_{i}(\FF)}\Phi(z) \tfrac{dz}{\Im z}\right| \cdot \left|\int_{\Cc_n} \omega_{i}(\FF)\right|\right)\rightarrow 0,\quad n\rightarrow \infty, $$
which finishes the proof in view of Lemma \ref{lem:Zelditch}. Finally, we remark that if the equidistribution (\ref{eq:equidigeodesics}) holds with an effective error-term then the above argument yields effective equidistribution with the same rate. 
\end{proof}
In the non-compact case, special care has to be given to the behavior at the cusp. This is ensured by the non-escape of mass as given in Proposition \ref{prop:nonescape} and the use of Poincar\'{e} series. 
\begin{proof}[Proof of Theorem \ref{thm:main}, non-compact case]
We may assume that $\Gamma$ has a cusp at $\infty$ of width $1$. Consider first the case $\Psi\in C_c^\infty(Y_\Gamma)$ with $\langle \Psi,1\rangle=0$. By Lemma \ref{lem:Poincare} we can express $\Psi$ as a sum of Poincar\'{e} series orthogonal to the constant function $1$ converging in the uniform norm. Thus by a standard approximation argument we are reduced to proving 
\begin{equation}\mu_n(P_{m}(\cdot| \psi))\rightarrow 0,\quad n\rightarrow \infty,\end{equation}
for $m\in \Z$ and $\psi\in C_c^\infty(\R_{>0})$ (fixed) with $\langle P_{m}(\cdot| \psi),1\rangle=0$. By conjugation we may assume that $m\geq 0$. Lemma \ref{lem:ODE} now yields:
\begin{align}
\mu_n(P_{m}(\cdot| \psi))=\mu_n(L_2 P_{m,2}(\cdot| \varphi)),\end{align}
with $\varphi:\R_{>0}\rightarrow \C$ rapidly decaying at $0$ and $\infty$, which in turn by Lemma \ref{lem:Stokes} equals 
\begin{align}\frac{1}{\vol(\PP(\Cc_n,\FF))}\left( \int_{\Cc_n} P_{m,2}(z| \varphi)\tfrac{dz}{\Im z}-  \sum_{i=1}^{2g} \left(\int_{\CC_{i}(\FF)}P_{m,2}(z| \varphi) \tfrac{dz}{\Im z}\right) \, \int_{\Cc_n} \omega_{i}(\FF)\right). 
\end{align}
By the same argument as in the compact case this converges to zero by the equidistribution of $\Cc_n$ as $n\rightarrow \infty$ in view of Lemma \ref{lem:Zelditch}. 

Now consider a general smooth and compactly supported function $\Psi\in C_c^\infty(Y_\Gamma)$. Let $Y>0$ and recall the definition of the bulk $\FF(Y)$ as defined in eq. (\ref{eq:bulk}) which we will identify with its projection to $Y_\Gamma$. The strategy is to reduce to the above treated case. Let $\psi_Y: Y_\Gamma \rightarrow [0,1]$ be a smooth approximation to the indicator function $\mathbf{1}_{\FF(Y)}$ so that $\psi_Y(z)=1$ for $z\in \FF(Y)$ and $\psi_Y(z)=0$ for $z\notin \overline{\FF(Y+1)}$ (assuming here that $Y$ is sufficiently large so that the cuspidal regions $\FF_\mathfrak{a}(Y)$ are all horostrips as in (\ref{eq:FFHH})) and write
\begin{align}
\mu_n(\Psi)=\mu_n\left(\Psi- \psi_Y\frac{\langle\Psi,1\rangle }{\langle\psi_Y,1\rangle }\right)+\langle\Psi,1\rangle \frac{\mu_n(\psi_Y) }{\langle\psi_Y,1\rangle }.
\end{align}
To deal with the second term we observe that by monotonicity  
$$\vol(\FF(Y))\leq \langle\psi_Y,1\rangle\leq \vol(\FF(Y+1))$$
which together with eq. (\ref{eq:bulkvol}) implies that 
$$\langle\psi_Y,1\rangle\rightarrow \vol(\Gamma),\quad \text{as } Y\rightarrow \infty.$$ Similarly by monotonicity and since $\mu_n$ is a probability measure we have 
\begin{align*}
1-\sum_{\mathfrak{a}\in \mathrm{cusp}(\FF)}\mu_n(\FF_\mathfrak{a}(Y))=\mu_n(\FF(Y))\leq \mu_n(\psi_Y)&\leq \mu_n(\FF(Y+1))\\
&=1-\sum_{\mathfrak{a}\in  \mathrm{cusp}(\FF)}\mu_n(\FF_\mathfrak{a}(Y+1)),\end{align*}
which combined with  non-escape of mass from Proposition \ref{prop:nonescape} gives that for each $\eps>0$ there exists $Y=Y(\eps)>0$ and $n_0=n_0(\eps)\geq 1$ (depending also on $\FF$) such that
$$\left|\frac{\mu_n(\psi_Y) }{\langle\psi_Y,1\rangle }-\frac{1}{\vol(\Gamma)}\right|\leq \eps,\quad n\geq n_0.  $$
 Since $\Psi- \psi_Y\frac{\langle\Psi,1\rangle }{\langle\psi_Y,1\rangle } $ is smooth, compactly supported and orthogonal to the constant function we conclude by the above that for $n$ sufficiently large (depending on $\eps$ through $Y=Y(\eps)$) the contribution is less than $\eps$, say. Since $\eps>0$ was arbitrary this finishes the proof in view of Lemma \ref{lem:Zelditch}.  
  \end{proof}
  \begin{remark}\label{remark:effective}
Obtaining an effective version of Theorem \ref{thm:main} in the non-co-compact case leads to subtleties at the cusps. We note however that in \cite{DuImTo} and \cite{HumphriesNordentoft22} equidistribution statements for partial coverings were obtained with explicit error-terms (power saving in the discriminant) in the non-compact case. Similarly, by a regularization process as in \cite[Sec. 5.3]{HumphriesNordentoft22} this \emph{can} be proved with the techniques of this paper under the assumption that the non-constant Weyl sums (for the spectral basis of residual and cuspidal Maa{\ss} forms and Eisenstein series) tends effectively to zero with polynomial dependence on the spectral parameter. This is however a stronger notion/assumption of equidistribution than the one used in Definition \ref{def:equid}.   \end{remark}
 \section{Applications}\label{sec:appl} In this section we outline some applications of Theorem \ref{thm:main}. In light of Lemma \ref{lem:trivial} we want examples of (sequences of) collections of geodesics $(\Cc_n)_{n\geq 1}$ that on the one hand, equidistribute in the unit tangent as $n\rightarrow \infty$ and on the other hand, do \emph{not} satisfy that $\CC\in \Cc_n\Rightarrow \overline{\CC}\in \Cc_n$ (where $\overline{\CC}$ denotes $\CC$ equipped with the opposite orientation). There exists by now many such examples in the literature as it falls under the topic of ``sparse equidistribution'' as outlined in the ICM adress by Michel and Venkatesh \cite{MichelVenk06}. 
\subsection{Sparse equidistribution of closed geodesics}\label{sec:sparse}
When quoting the literature below we are throughout using the analogue of Lemma \ref{lem:Zelditch} for the unit tangent bundle which ensures that it suffices to check the convergence (\ref{eq:geoequid}) for smooth and compactly supported functions $\Psi:\Gamma\backslash G\rightarrow \C$, see \cite[Cor.\ 6.5]{Zelditch92} for details. The first example we will consider is the following sparse equidistribution result applicable for any discrete and cofinite subgroup $\Gamma\leq \PSL_2(\R)$.
\begin{thm}\label{thm:intro3}
Fix $\eps>0$. Let $\Gamma\leq \PSL_2(\R)$ be a discrete and cofinite subgroup.  Let $L_1<L_2<\ldots$  be a sequence of real numbers such that  $L_n\rightarrow \infty$ as $n\rightarrow \infty$. For each $n\geq 1$, let
$$\Cc_n\subset\{\CC\subset Y_\Gamma: |\CC|\leq L_n\}.$$   
be a ``not too thin subset'' of the packet of primitive oriented closed geodesics of length $\leq L_n$, in the sense that 
$$ |\Cc_n|\geq \eps \cdot |\{\CC\subset Y_\Gamma: |\CC|\leq L_n\}|. $$ 
Then $\Cc_1,\Cc_2,\ldots$ equidistribute  in the unit tangent bundle in the sense of Definition \ref{def:equid}.
\end{thm}
\begin{proof}[Proof sketch]
This follows from the ergodicity of the geodesic flow as employed in \cite{AkaEinsiedler16} combined with the equidistrbution of geodesics of bounded length as in e.g. \cite[Cor.\ 6.5]{Zelditch92}, see \cite[Thm.\ 4.7]{ConstantinescuNordentoft24} for the complete argument. 
\end{proof}

We will now go on to describe equidistribution results for Fuchsian groups of arithmetic origin.  In this setting the closed geodesics fit into the framework of \emph{periodic torus orbits} and have been extensively studied from many perspectives, see e.g. \cite{EinLindMichVenk09}. The arithmetic structure allows one to obtain equidistribution for much sparser collections of closed geodesics. Firstly we have the following version for subcollections of the length packets for the modular group.

\begin{thm}[Einsiedler--Lindenstrauss--Michel--Venkatesh]\label{thm:intro1}
Let $\FF$ be a fundamental polygon for the modular surface $Y_0(1):=\PSL_2(\Z)\backslash \Hb$.  Let $\ell_1<\ell_2<\ldots $ be the primitive length spectrum of $Y_0(1)$. For $n\geq 1$, let
$$\Cc_n\subset \{\CC\subset Y_0(1): |\CC|=\ell_n\},
$$ 
be a ``not too thin subset'' of the length packet, in the sense that 
$$ \frac{\log |\Cc_n|}{\log (\sum_{\CC: |\CC|=\ell_n}|\CC|))}\rightarrow 1,\quad n\rightarrow \infty. $$ 
Then $\Cc_1,\Cc_2,\ldots$ equidistribute in the unit tangent bundle in the sense of Definition \ref{def:equid}.
\end{thm}
\begin{proof}
This follows from the results of \cite{EinLindMichVenk12} as explained in the discussion on page 2 of \cite{AkaEinsiedler16}. Note that in \cite{EinLindMichVenk12} they consider  closed geodesics $\CC_A$ associated to elements $A\in \Cl(\mathcal{O}_D)$ of class groups of orders in (real) quadratic fields of discriminant $D\rightarrow \infty$. It is however a classical fact \cite{Sarnak82} that all closed geodesics on $Y_0(1)$ can be obtained in this way and that $|\CC_A|\rightarrow \infty$ for $A\in \Cl(\mathcal{O}_D)$ as $D\rightarrow \infty$. 
\end{proof}
Let $M$ be a square-free integer and let $Y_0(M):=\Gamma_0(M)\backslash \Hb$ denote the modular surface of level $M$. Let $D>0$ be a fundamental discriminant such that all prime divisors of $M$ split in $\Q(\sqrt{D})$. Then there exists an optimal embedding $\Q(\sqrt{D})\hookrightarrow M_2(\Q)$ which gives rise to a map
$$\Cl^+_D\rightarrow \{\text{primitive oriented closed geodesics on $Y_0(M)$} \},\quad A\mapsto \CC_A(M),$$ 
satisfying $|\CC_A(M)|=2\log \epsilon_D$ where $\epsilon_D\in \Q(\sqrt{D})$ denotes the fundamental totally positive unit (for details consult e.g. \cite[Sec. 6]{Popa06}). We have the following general equidistribution for cosets of sufficiently large subgroups.
\begin{thm}[Waldspurger, Michel--Harcos, Popa]\label{thm:intro2}
Fix $\delta \in [0,\frac{1}{2826})$. Let $\FF$ be a fundamental polygon for the modular surface $Y_0(M)$ of square-free level $M$. Let $(D_n)_{n\geq 1}$ be a sequence of distinct positive fundamental discriminants such that all prime divisors of $M$ split in $\Q(\sqrt{D_n})$. For $n\geq 1$, let
$$\Cc_n=\{\CC_A(M)\subset Y_0(M): A\in BH_n\},
$$ 
be the collection of primitive oriented closed geodesics on $Y_0(M)$ associated to the coset $BH_n$ of a subgroup $H_n\leq  \Cl_{D_n}^+$ of the narrow class group of $\Q(\sqrt{D_n})$ of index $\leq D_n^{\delta}$. Then $\Cc_1,\Cc_2,\ldots$ equidistribute in the unit tangent bundle in the sense of Definition \ref{def:equid}.
\end{thm}
\begin{proof}
This follows by a straightforward extension of the results of Popa to automorphic forms of general weight and we will simply indicate the necessary changes. One applies Weyl's criterion: to bound the Weyl sums one combines an explicit version of the Waldspurger formula \cite{Waldspurger85}, as obtained by Popa \cite{Popa06} for weight 0, and the subconvexity bounds due to Michel \cite[Thm.\ 2]{Michel04} and Harcos--Michel\cite[Thm.\ 1]{HarcosMichel06} for the discrete Weyl sums and due to Blomer--Harcos--Michel \cite[Thm.\ 2]{BlHaMi07} for the continuous Weyl sums. The required explicit version of Waldspurger's formula in the higher weight case follows by combining the work of Martin--Whitehouse \cite[Thm. 4.2]{MartinWhitehouse09} with certain local computations. This has been carried out for weight $2$ automorphic forms by Peter Humphries and the first named author, see \cite[Lemma 6.14]{HumphriesNordentoft22}, and the same argument applies in general with minor modifications.  
\end{proof} 

Combining all of the above we obtain the claimed applications from the introduction. 

\begin{proof}[Proof of Corollaries \ref{cor:intro3}, \ref{cor:intro1}, \ref{cor:intro2}, and \ref{cor:intro4}]
This follows directly by combining Theorem \ref{thm:main} with respectively, Theorems \ref{thm:intro3}, \ref{thm:intro1}, \ref{thm:intro2} and \cite[Thm. 1.1]{Nordentoft23}. 
\end{proof}
\subsection{The results of Duke--Imamo\={g}lu--T\'{o}th and Humphries--Nordentoft}
Finally, we will explain how to deduce (some of) the equidistribution results in \cite{DuImTo} and \cite{HumphriesNordentoft22} from the above. Let either $q=1$ or $q$ be prime and let $D$ be a positive fundamental discriminant for which $q$ splits in $\Q(\sqrt{D})$ if $q > 1$. Recall that \cite[Sec.\ 3]{DuImTo} and \cite[Sec.\ 3]{HumphriesNordentoft22} associated to each $A\in \Cl_D^+$ a thin subgroup $\Gamma_A(q)\subset \PSL_2(\Z)$ and a fundamental polygon $\FF_A(q)\subset \Hb$. Using the convention (\ref{eq:BBsubset}) we will think of $\FF_A(q)$ as a partial covering of $Y_0(q)=\Gamma_0(q) \backslash \Hb$. We can now translate the results of this paper to the equidistribution results in \emph{loc.\ cit.} using here the formulation in terms of continuity sets. 

\begin{cor}[Cf.\ {\cite[Thm.\ 2]{DuImTo}, \cite[Thm.\ 1.3]{HumphriesNordentoft22}}]\label{cor:itfollows}
Fix $\delta \in [0,\frac{1}{2826})$ and fix either $q = 1$ or $q$ an odd prime. For each positive fundamental discriminant $D$ for which $q$ splits in $\Q(\sqrt{D})$ if $q > 1$, choose a coset $CH$ with $C \in \Cl_D^+$ and $H = H_D$ a subgroup of $\Cl_D^+$ of index $\leq D^{\delta}$. Then for each fixed continuity set $\BB \subset Y_0(q)$ it holds that
\begin{equation}
\label{eqn:subgroupequidistribution}
\frac{\sum_{A \in CH} \vol(\FF_A(q) \cap \BB)}{\sum_{A \in CH} \vol(\FF_A(q))} \rightarrow \frac{\vol(\BB)}{\vol(\Gamma_0(q))},\quad \text{as }D\rightarrow \infty.
\end{equation}
\end{cor}
\begin{proof}
By Lemma \ref{lem:B} we are reduced to showing the convergence (\ref{eq:coverevenlydef}) for continuous and bounded functions $\phi:Y_\Gamma\rightarrow \C$. Let $\FF$ equal the fundamental polygon $\FF_0$ defined in (\ref{eq:Fstd}) when $q=1$ and a special fundamental polygon $\FF_q$ when  $q$ is prime. Then by Lemmas \ref{lem:DITequiv} and \ref{lem:HuNoequiv} we have for each $A\in \Cl_D^+$ as above that
$$ \int_{\FF_A(q)}\phi=  \int_{\PP(\CC_A(q),\FF)} \phi +m\int_{Y_0(q)} \phi,  $$
for some integer $m\geq 0$. Now the result follows from Theorem \ref{thm:main} in view of Theorem \ref{thm:intro2} using the standard fact that if $\frac{a_n}{b_n}\rightarrow \frac{a}{b}$ as $n\rightarrow \infty$ with $b_n,b>0$ then for any choice $m_n\geq 1$ it holds that $\frac{a_n+m_na}{b_n+m_n b}\rightarrow \frac{a}{b}$ as $n\rightarrow \infty$.
\end{proof}

\bibliography{mybib}

@article{SeegerSogge,
	author = {Seeger, Andreas and Sogge, Christopher D.},
	doi = {10.1512/iumj.1989.38.38031},
	fjournal = {Indiana University Mathematics Journal},
	issn = {0022-2518},
	journal = {Indiana Univ. Math. J.},
	language = {English},
	number = {3},
	pages = {669--682},
	title = {Bounds for eigenfunctions of differential operators},
	volume = {38},
	year = {1989},
	zbl = {0703.35133},
	zbmath = {4153057}}

@article{NordentoftConcentration,
	author = {{Nordentoft}, Asbj{\o}rn Christian},
	doi = {10.1017/fms.2023.85},
	fjournal = {Forum of Mathematics, Sigma},
	issn = {2050-5094},
	journal = {Forum Math. Sigma},
	language = {English},
	note = {Id/No e91},
	pages = {38},
	title = {Concentration of closed geodesics in the homology of modular curves},
	volume = {11},
	year = {2023},
	zbl = {1540.11051},
	zbmath = {7753226}}

@article{JaasaariLesterSaha23,
	author = {J{\"a}{\"a}saari, Jesse and Lester, Stephen and Saha, Abhishek},
	doi = {10.1007/s00039-024-00690-x},
	fjournal = {Geometric and Functional Analysis. GAFA},
	issn = {1016-443X},
	journal = {Geom. Funct. Anal.},
	language = {English},
	number = {5},
	pages = {1460--1532},
	title = {Mass equidistribution for {Saito}-{Kurokawa} lifts},
	volume = {34},
	year = {2024},
	zbmath = {7921924}}

@article{ConstantinescuNordentoft24,
	author = {Constantinescu, Petru and Nordentoft, Asbj{\o}rn Christian},
	date = {2025/08/01},
	doi = {10.1007/s00039-025-00715-z},
	id = {Constantinescu2025},
	isbn = {1420-8970},
	journal = {Geometric and Functional Analysis},
	number = {4},
	pages = {1108--1146},
	title = {Non-vanishing of Geodesic Periods of Automorphic Forms},
	url = {https://doi.org/10.1007/s00039-025-00715-z},
	volume = {35},
	year = {2025}}

@article{HumphriesRadziwill22,
	author = {Humphries, Peter and Radziwi{\l}{\l}, Maksym},
	doi = {10.1002/cpa.22076},
	fjournal = {Communications on Pure and Applied Mathematics},
	issn = {0010-3640},
	journal = {Commun. Pure Appl. Math.},
	language = {English},
	number = {9},
	pages = {1936--1996},
	title = {Optimal small scale equidistribution of lattice points on the sphere, {Heegner} points, and closed geodesics},
	volume = {75},
	year = {2022},
	zbl = {1504.11084},
	zbmath = {7597045}}

@article{ErlandssonSouto23,
	author = {Erlandsson, Viveka and Souto, Juan},
	doi = {10.1017/etds.2021.166},
	fjournal = {Ergodic Theory and Dynamical Systems},
	issn = {0143-3857},
	journal = {Ergodic Theory Dyn. Syst.},
	language = {English},
	number = {3},
	pages = {887--903},
	title = {Distribution in the unit tangent bundle of the geodesics of given type},
	volume = {43},
	year = {2023},
	zbl = {1517.37040},
	zbmath = {7680109}}

@article{ErlandssonSouto22,
	author = {Erlandsson, Viveka and Souto, Juan},
	doi = {10.1093/imrn/rnad156},
	fjournal = {IMRN. International Mathematics Research Notices},
	issn = {1073-7928},
	journal = {Int. Math. Res. Not.},
	language = {English},
	number = {13},
	pages = {10298--10318},
	title = {Counting and equidistribution of reciprocal geodesics and dihedral groups},
	volume = {2024},
	year = {2024},
	zbmath = {7935865}}

@article{EinLindMichVenk09,
	author = {Einsiedler, Manfred and Lindenstrauss, Elon and Michel, Philippe and Venkatesh, Akshay},
	doi = {10.1215/00127094-2009-023},
	fjournal = {Duke Mathematical Journal},
	issn = {0012-7094},
	journal = {Duke Math. J.},
	language = {English},
	number = {1},
	pages = {119--174},
	title = {Distribution of periodic torus orbits on homogeneous spaces},
	volume = {148},
	year = {2009},
	zbl = {1172.37003},
	zbmath = {5555678}}

@article{Sarnak82,
	author = {Sarnak, Peter},
	doi = {https://doi.org/10.1016/0022-314X(82)90028-2},
	issn = {0022-314X},
	journal = {Journal of Number Theory},
	number = {2},
	pages = {229-247},
	title = {Class numbers of indefinite binary quadratic forms},
	url = {https://www.sciencedirect.com/science/article/pii/0022314X82900282},
	volume = {15},
	year = {1982}}

@incollection{SarnakReciprocal,
	author = {Sarnak, Peter},
	booktitle = {Analytic number theory. A tribute to Gauss and Dirichlet. Proceedings of the Gauss-Dirichlet conference, G\"ottingen, Germany, June 20--24, 2005},
	isbn = {978-0-8218-4307-9},
	language = {English},
	pages = {217--237},
	publisher = {Providence, RI: American Mathematical Society (AMS)},
	title = {Reciprocal geodesics},
	year = {2007},
	zbl = {1198.11039},
	zbmath = {5233965}}

@article{MartinWhitehouse09,
	author = {Martin, Kimball and Whitehouse, David},
	doi = {10.1093/imrn/rnn127},
	fjournal = {IMRN. International Mathematics Research Notices},
	issn = {1073-7928},
	journal = {Int. Math. Res. Not.},
	language = {English},
	number = {1},
	pages = {141--191},
	title = {Central {{\(L\)}}-values and toric periods for {{\(\text{GL}(2)\)}}},
	volume = {2009},
	year = {2009},
	zbl = {1193.11046},
	zbmath = {5495097}}

@article{AkaEinsiedler16,
	author = {Aka, Menny and Einsiedler, Manfred},
	doi = {10.1017/etds.2014.68},
	fjournal = {Ergodic Theory and Dynamical Systems},
	issn = {0143-3857},
	journal = {Ergodic Theory Dyn. Syst.},
	language = {English},
	number = {2},
	pages = {335--342},
	title = {Duke's theorem for subcollections},
	volume = {36},
	year = {2016},
	zbl = {1355.37004},
	zbmath = {6585342}}

@book{Zelditch92,
	author = {Zelditch, Steven},
	doi = {10.1090/memo/0465},
	fseries = {Memoirs of the American Mathematical Society},
	isbn = {978-0-8218-2526-6; 978-1-4704-0891-6},
	issn = {0065-9266},
	language = {English},
	publisher = {Providence, RI: American Mathematical Society (AMS)},
	series = {Mem. Am. Math. Soc.},
	title = {Selberg trace formulae and equidistribution theorems for closed geodesics and {Laplace} eigenfunctions: {Finite} area surfaces},
	volume = {465},
	year = {1992},
	zbl = {0753.11023},
	zbmath = {32356}}

@article{Bowen72,
	author = {Bowen, Rufus},
	doi = {10.2307/2374628},
	fjournal = {American Journal of Mathematics},
	issn = {0002-9327},
	journal = {Am. J. Math.},
	language = {English},
	pages = {413--423},
	title = {The equidistribution of closed geodesics},
	volume = {94},
	year = {1972},
	zbl = {0249.53033},
	zbmath = {3394998}}

@article{Nordentoft23,
	adsurl = {https://ui.adsabs.harvard.edu/abs/2023arXiv230518625N},
	archiveprefix = {arXiv},
	author = {{Nordentoft}, Asbj{\o}rn Christian},
	doi = {10.48550/arXiv.2305.18625},
	eid = {arXiv:2305.18625},
	eprint = {2305.18625},
	journal = {arXiv e-prints},
	month = may,
	pages = {arXiv:2305.18625},
	primaryclass = {math.NT},
	title = {{Equidistribution of $q$-orbits of closed geodesics}},
	year = 2023}

@article{Skubenko62,
	author = {Skubenko, Boris F.},
	fjournal = {Izvestiya Akademii Nauk SSSR. Seriya Matematicheskaya},
	issn = {0373-2436},
	journal = {Izv. Akad. Nauk SSSR, Ser. Mat.},
	language = {Russian},
	pages = {721--752},
	title = {Die asymptotische {Verteilung} der {Gitterpunkte} auf einem einschaligen {Hyperboloid} und {Ergodens{\"a}tze}},
	volume = {26},
	year = {1962},
	zbl = {0107.04201},
	zbmath = {3174568}}

@article{FouvryKowMich15,
	author = {Fouvry, {\'E}tienne and Kowalski, Emmanuel and Michel, Philippe},
	doi = {10.1007/s00039-015-0310-2},
	fjournal = {Geometric and Functional Analysis. GAFA},
	issn = {1016-443X},
	journal = {Geom. Funct. Anal.},
	language = {English},
	number = {2},
	pages = {580--657},
	title = {Algebraic twists of modular forms and {Hecke} orbits},
	volume = {25},
	year = {2015},
	zbl = {1344.11036},
	zbmath = {6438346}}

@article{HumphriesNordentoft22,
	author = {{Humphries}, Peter and {Nordentoft}, Asbj{\o}rn Christian},
	journal = {J. Eur. Math. Soc., published online first},
	title = {{Sparse Equidistribution of Geometric Invariants of Real Quadratic Fields}},
	year = 2025}

@book{Hatcher02,
	author = {Hatcher, Allen},
	isbn = {0-521-79540-0},
	language = {English},
	publisher = {Cambridge: Cambridge University Press},
	title = {Algebraic topology},
	year = {2002},
	zbl = {1044.55001},
	zbmath = {2103273}}

@article{Kulkarni91,
	author = {Ravi S. {Kulkarni}},
	doi = {10.2307/2374900},
	fjournal = {{American Journal of Mathematics}},
	issn = {0002-9327},
	journal = {{Am. J. Math.}},
	language = {English},
	msc2010 = {11F06 11B57 20F05},
	number = {6},
	pages = {1053--1133},
	publisher = {Johns Hopkins University Press, Baltimore, MD},
	title = {{An arithmetic-geometric method in the study of the subgroups of the modular group}},
	volume = {113},
	year = {1991},
	zbl = {0758.11024}}

@book{Beardon83,
	author = {Alan F. {Beardon}},
	fjournal = {{Graduate Texts in Mathematics}},
	issn = {0072-5285},
	journal = {{Grad. Texts Math.}},
	language = {English},
	msc2010 = {30-02 20-02 30F35 20H10 30F40 14L30 14L35 22E40},
	publisher = {Springer, Cham},
	title = {{The geometry of discrete groups}},
	volume = {91},
	year = {1983},
	zbl = {0528.30001}}

@incollection{Darmon94,
	author = {Darmon, Henri},
	booktitle = {Elliptic curves and related topics},
	doi = {10.1215/s0012-7094-94-07604-7},
	mrclass = {11G40 (11F67 11G05)},
	mrnumber = {1260954},
	mrreviewer = {Glenn Stevens},
	pages = {45--59},
	publisher = {Amer. Math. Soc., Providence, RI},
	series = {CRM Proc. Lecture Notes},
	title = {Heegner points, {H}eegner cycles, and congruences},
	url = {https://doi.org/10.1215/s0012-7094-94-07604-7},
	volume = {4},
	year = {1994}}

@article{Michel04,
	author = {Michel, Philippe},
	doi = {10.4007/annals.2004.160.185},
	fjournal = {Annals of Mathematics. Second Series},
	issn = {0003-486X},
	journal = {Ann. of Math. (2)},
	mrclass = {11F66 (11G18)},
	mrnumber = {2119720},
	mrreviewer = {Gergely Harcos},
	number = {1},
	pages = {185--236},
	title = {The subconvexity problem for {R}ankin-{S}elberg {$L$}-functions and equidistribution of {H}eegner points},
	url = {https://doi.org/10.4007/annals.2004.160.185},
	volume = {160},
	year = {2004}}

@article{Waldspurger85,
	author = {Waldspurger, Jean-Loup},
	fjournal = {Compositio Mathematica},
	issn = {0010-437X},
	journal = {Compositio Math.},
	mrclass = {11F70 (11F67 22E55)},
	mrnumber = {783511},
	mrreviewer = {Stephen Gelbart},
	number = {2},
	pages = {173--242},
	title = {Sur les valeurs de certaines fonctions {$L$} automorphes en leur centre de sym\'{e}trie},
	url = {http://www.numdam.org/item?id=CM_1985__54_2_173_0},
	volume = {54},
	year = {1985}}

@article{LuoSarnak95,
	author = {Luo, Wen Zhi and Sarnak, Peter},
	fjournal = {Institut des Hautes \'{E}tudes Scientifiques. Publications Math\'{e}matiques},
	issn = {0073-8301},
	journal = {Inst. Hautes \'{E}tudes Sci. Publ. Math.},
	mrclass = {11F72 (11M41 58F19)},
	mrnumber = {1361757},
	mrreviewer = {Yiannis Petridis},
	number = {81},
	pages = {207--237},
	title = {Quantum ergodicity of eigenfunctions on {${\rm PSL}_2(\bold Z)\backslash \bold H^2$}},
	url = {http://www.numdam.org/item?id=PMIHES_1995__81__207_0},
	year = {1995}}

@article{EinLindMichVenk12,
	author = {Einsiedler, Manfred and Lindenstrauss, Elon and Michel, Philippe and Venkatesh, Akshay},
	doi = {10.4171/LEM/58-3-2},
	fjournal = {L'Enseignement Math\'{e}matique. Revue Internationale. 2e S\'{e}rie},
	issn = {0013-8584},
	journal = {Enseign. Math. (2)},
	mrclass = {11F11 (11E16 11F37 14C22 37A45 53C22)},
	mrnumber = {3058601},
	mrreviewer = {Thomas R. Shemanske},
	number = {3-4},
	pages = {249--313},
	title = {The distribution of closed geodesics on the modular surface, and {D}uke's theorem},
	url = {https://doi.org/10.4171/LEM/58-3-2},
	volume = {58},
	year = {2012}}

@incollection{MichelVenk06,
	author = {Michel, Philippe and Venkatesh, Akshay},
	booktitle = {International {C}ongress of {M}athematicians. {V}ol. {II}},
	mrclass = {11F66 (11F67 11M41 37A45)},
	mrnumber = {2275604},
	mrreviewer = {Gergely Harcos},
	pages = {421--457},
	publisher = {Eur. Math. Soc., Z\"{u}rich},
	title = {Equidistribution, {$L$}-functions and ergodic theory: on some problems of {Y}u. {L}innik},
	year = {2006}}

@article{BlHaMi07,
	author = {Blomer, Valentin and Harcos, Gergely and Michel, Philippe},
	doi = {10.1016/j.ansens.2007.05.003},
	fjournal = {Annales Scientifiques de l'\'{E}cole Normale Sup\'{e}rieure. Quatri\`eme S\'{e}rie},
	issn = {0012-9593},
	journal = {Ann. Sci. \'{E}cole Norm. Sup. (4)},
	mrclass = {11F66 (11R42)},
	mrnumber = {2382859},
	mrreviewer = {Emmanuel P. Royer},
	number = {5},
	pages = {697--740},
	title = {Bounds for modular {$L$}-functions in the level aspect},
	url = {https://doi.org/10.1016/j.ansens.2007.05.003},
	volume = {40},
	year = {2007}}

@article{Popa06,
	author = {Popa, Alexandru},
	doi = {10.1112/S0010437X06002259},
	fjournal = {Compositio Mathematica},
	issn = {0010-437X},
	journal = {Compos. Math.},
	mrclass = {11F67 (11F27 11F70)},
	mrnumber = {2249532},
	mrreviewer = {Gergely Harcos},
	number = {4},
	pages = {811--866},
	title = {Central values of {R}ankin {$L$}-series over real quadratic fields},
	url = {https://doi.org/10.1112/S0010437X06002259},
	volume = {142},
	year = {2006}}

@article{HarcosMichel06,
	author = {Harcos, Gergely and Michel, Philippe},
	doi = {10.1007/s00222-005-0468-6},
	fjournal = {Inventiones Mathematicae},
	issn = {0020-9910},
	journal = {Invent. Math.},
	mrclass = {11F66 (11F67 11M41)},
	mrnumber = {2207235},
	mrreviewer = {K. Soundararajan},
	number = {3},
	pages = {581--655},
	title = {The subconvexity problem for {R}ankin-{S}elberg {$L$}-functions and equidistribution of {H}eegner points. {II}},
	url = {https://doi.org/10.1007/s00222-005-0468-6},
	volume = {163},
	year = {2006}}

@book{Li68,
	author = {Linnik, Yuri},
	mrclass = {10.65},
	mrnumber = {0238801},
	pages = {ix+192},
	publisher = {Springer-Verlag New York Inc., New York},
	series = {Translated from the Russian by M. S. Keane. Ergebnisse der Mathematik und ihrer Grenzgebiete, Band 45},
	title = {Ergodic properties of algebraic fields},
	year = {1968}}

@article{DuImTo,
	author = {Duke, William and Imamo\={g}lu, \"{O}zlem and T\'{o}th, \'{A}rp\'{a}d},
	doi = {10.4007/annals.2016.184.3.8},
	fjournal = {Annals of Mathematics. Second Series},
	issn = {0003-486X},
	journal = {Ann. of Math. (2)},
	mrclass = {11F11 (11R11 14H55)},
	mrnumber = {3549627},
	mrreviewer = {Benjamin Linowitz},
	number = {3},
	pages = {949--990},
	title = {Geometric invariants for real quadratic fields},
	url = {https://doi.org/10.4007/annals.2016.184.3.8},
	volume = {184},
	year = {2016}}

@article{Du88,
	author = {Duke, William},
	doi = {10.1007/BF01393993},
	fjournal = {Inventiones Mathematicae},
	issn = {0020-9910},
	journal = {Invent. Math.},
	mrclass = {11F11 (11E32 11F30 11F37)},
	mrnumber = {931205},
	mrreviewer = {Mark Sheingorn},
	number = {1},
	pages = {73--90},
	title = {Hyperbolic distribution problems and half-integral weight {M}aass forms},
	url = {https://doi.org/10.1007/BF01393993},
	volume = {92},
	year = {1988}}

@article{DuFrIw02,
	author = {Duke, William and Friedlander, John B. and Iwaniec, Henryk},
	doi = {10.1007/s002220200223},
	fjournal = {Inventiones Mathematicae},
	issn = {0020-9910},
	journal = {Invent. Math.},
	mrclass = {11F66 (11F30 11M41 11R29)},
	mrnumber = {1923476},
	mrreviewer = {K. Soundararajan},
	number = {3},
	pages = {489--577},
	title = {The subconvexity problem for {A}rtin {$L$}-functions},
	url = {https://doi.org/10.1007/s002220200223},
	volume = {149},
	year = {2002}}

@book{Iw,
	author = {Iwaniec, Henryk},
	doi = {10.1090/gsm/053},
	edition = {Second},
	isbn = {0-8218-3160-7},
	mrclass = {11F72 (11F12 11F37)},
	mrnumber = {1942691},
	pages = {xii+220},
	publisher = {American Mathematical Society, Providence, RI; Revista Matem\'atica Iberoamericana, Madrid},
	series = {Graduate Studies in Mathematics},
	title = {Spectral methods of automorphic forms},
	url = {https://doi.org/10.1090/gsm/053},
	volume = {53},
	year = {2002}}

@article{PeRi,
	author = {Petridis, Yiannis N. and Risager, Morten S.},
	doi = {10.1007/s00222-017-0784-7},
	fjournal = {Inventiones Mathematicae},
	issn = {0020-9910},
	journal = {Invent. Math.},
	mrclass = {11F67 (11E45 11M36)},
	mrnumber = {3802302},
	number = {3},
	pages = {997--1053},
	title = {Arithmetic statistics of modular symbols},
	url = {https://doi.org/10.1007/s00222-017-0784-7},
	volume = {212},
	year = {2018}}
\bibliographystyle{amsplain}
\end{document}